\documentclass[a4paper]{article}%Test for
\usepackage{amsmath,amssymb,amsfonts,amsthm} % 引入AMS的数学、符号、字体和定理宏包
\usepackage{mathrsfs} % 引入花体字母宏包
\usepackage{enumerate} % 用于创建自定义的列表环境
\usepackage{cases} % 用于创建多行方程中的分段函数
\usepackage{float} % 提供更多控制浮动对象的位置的选项
\usepackage{graphicx} % 引入图形宏包 Include this line if your document contains figures,
\allowdisplaybreaks[4] % 允许多行公式在页面之间断开
\usepackage{wrapfig}  % 允许图形和文字环绕
\usepackage{url} % 支持URL格式
\usepackage[usenames]{color} % 引入颜色宏包，并使用预定义的颜色名称
\usepackage{setspace} % 设置行间距

\def\b1{\mbox{\boldmath $1$}} % 定义加粗的数字1

\usepackage{geometry}
\makeatletter
\newcommand{\Biggg}{\bBigg@{3.5}} % 定义一个新的大尺寸符号
\makeatother

\theoremstyle{plain} % 设置定理样式为plain
\newtheorem{theorem}{\bf Theorem}[section] % 定义新的定理环境，编号基于章节
\newtheorem{corollary}{\bf Corollary}[section] % 定义新的推论环境，编号与定理共享
\newtheorem{lemma}{\bf Lemma}[section] % 定义新的引理环境，编号与定理共享
\newtheorem{proposition}{\bf Proposition}[section] % 定义新的命题环境，编号与定理共享
\newtheorem{definition}{\bf Definition}[section] % 定义新的定义环境，编号与定理共享
\newtheorem{remark}{\bf Remark}[section]  % 定义新的备注环境，编号与定理共享
\makeatother
\allowdisplaybreaks %  允许多行公式在页面之间断开.

\usepackage[numbers,sort&compress]{natbib} % 引入natbib宏包，使参考文献连续编号的引用以压缩方式出现
\usepackage[bookmarksnumbered,bookmarksopen,colorlinks,citecolor=blue,linkcolor=blue]{hyperref} % 引入超链接宏包，并设置超链接属性

\begin{document}
\date{}

\title{Open-loop and closed-loop solvabilities for zero-sum stochastic linear quadratic differential games of Markovian regime switching system
\thanks{Fan Wu and Xin Zhang are supported by the National Natural Science Foundation of China (Grant Nos. 12171086, 12371472), Tianyuan Fund for Mathematics grant 12426652, and Jiangsu Provincial Scientific Research Center of Applied Mathematics grant BK20233002. Xun Li is supported by the Research Grants Council of Hong Kong under grants 15221621, 15226922 and 15225124, PolyU 4-ZZVB, 4-ZZP4, 1-ZVXA.}
}
\author{Fan Wu \thanks{School of Big Data and Statistics, Anhui University, Hefei 230601, China}
\qquad Xun Li \thanks{Department of Applied Mathematics, Hong Kong Polytechnic University, Hong Kong, China}
\qquad Xin Zhang \thanks{Corresponding Author: School of Mathematics, Southeast University, Nanjing 211189, China; E-mail: x.zhang.seu@gmail.com}
}
\maketitle
\noindent{\bf Abstract:}
This paper investigates zero-sum stochastic linear quadratic (SLQ) differential games with Markovian jumps. Open-loop and closed-loop solvabilities are studied by employing a new ``decomposition method", which decomposes the open-loop and closed-loop solvability problems of zero-sum SLQ differential games into two coupled SLQ control problems. Under the uniform convexity-concavity condition, we construct the open-loop saddle point along with its closed-loop representation based on the solution to a system of constrained coupled differential Riccati equations (CDREs), whose solvability is also established by employing the dimension extension technique and the continuation method. Finally, we provide a concrete example and present its closed-form saddle point based on the theoretical results obtained. \medskip


\section{Introduction}\label{section-1}
\subsection{Basic notation and problem formulation}
 Let $(\Omega,\mathcal{F},\mathbb{F},\mathbb{P} )$ be a complete filtered probability space on which a one-dimensional standard Brownian motion $W(\cdot)$ and a continuous time, irreducible Markov chain $\alpha(\cdot)$ are defined, with $\mathbb{F}=\{\mathcal{F}_{t}\}_{t\geq 0}$ being its natural filtration augmented by all $\mathbb{P}$-null sets in $\mathcal{F}$. Throughout this paper, we denote the state space of the Markov chain $\alpha(\cdot)$ as $\mathcal{S}:=\left\{1,2,..., L\right\}$, where $L$ is a finite natural number. Let $\mathbb{R}^{n \times m}$ be the Euclidean space of all $n \times m  $ matrices, and $\mathbb{R}^{n}$ be the $\mathbb{R}^{n \times 1}$ for simplicity. Additionally, the set of all $n\times n$ symmetric matrices is denoted by $\mathbb{S}^{n}$. Specifically, the sets of all $ n\times n$ semi-positive definite matrices and positive definite matrices are denoted by $\overline{\mathbb{S}_{+}^n}$ and $\mathbb{S}_{+}^n$, respectively. Let $\mathcal{P}$ be the $\mathbb{F}$ predictable $\sigma$-field on $[0,T]\times\Omega$, and we write $\varphi \in \mathcal{P}$ (respectively, $\varphi \in \mathbb{F}$) if the stochastic process $\varphi(\cdot)$ is $\mathcal{P}$-measurable (respectively, $\mathbb{F}$-progressively measurable). Then, for any Euclidean space $\mathbb{H}$ and $t \in [0,T]$, we introduce the following spaces:
 	$$\begin{array}{l}
		C(t, T; \mathbb{H}) =\big\{\varphi:[t, T] \rightarrow \mathbb{H} \text{ }\big|\text{ } \varphi(\cdot) \text { is continuous }\big\}, \\[0.3cm]
		L^{\infty}(t, T ; \mathbb{H}) =\big\{\varphi:[t, T] \rightarrow \mathbb{H} \text{ }\big|\text{ } \operatorname{esssup}_{s \in[t, T]}| \varphi(s)|<\infty\big\},\\[0.3cm]
		L_{\mathcal{F}_{T}}^{2}(\Omega ; \mathbb{H}) =\big\{\xi: \Omega \rightarrow \mathbb{H} \text{ }\big|\text{ } \xi \text { is } \mathcal{F}_{T} \text {-measurable, } \mathbb{E}|\xi|^{2}<\infty\big\},\\[0.3cm]
		\mathcal{S}_{\mathbb{F}}^{2}(t, T ; \mathbb{H}) =\big\{\varphi:[t, T] \times \Omega \rightarrow \mathbb{H} \text{ }\big|\text{ }\varphi(\cdot) \in \mathbb{F} \text {, } \mathbb{E}\big[\sup _{s \in[t, T]}|\varphi(s)|^{2}\big]<\infty\big\}, \\[0.3cm]
		L_{\mathbb{F}}^{2}(t, T ; \mathbb{H}) =\big\{\varphi:[t, T] \times \Omega \rightarrow \mathbb{H} \text{ }\big|\text{ } \varphi(\cdot) \in \mathbb{F}\text{, } \mathbb{E} \int_{t}^{T}|\varphi(s)|^{2} ds<\infty\big\}, \\[0.3cm]
		 L_{\mathcal{P}}^{2}\left(t, T ; \mathbb{H}\right)=\big\{\varphi:[t, T] \times \Omega \rightarrow \mathbb{H} \text{ }\big|\text{ } \varphi(\cdot) \in \mathcal{P}\text{, } \mathbb{E} \int_{t}^{T}|\varphi(s)|^{2} ds<\infty\big\} .
	\end{array}$$

 Consider the following  controlled linear stochastic differential equation (SDE, for short) with Markovian jumps over the finite time horizon $[0, T]$:
 \begin{equation}\label{state}
   \left\{
   \begin{aligned}
   dX(t)&=\left[A\left(t,\alpha(t)\right)X(t)+B_{1}\left(t,\alpha(t)\right)u_{1}(t)+B_{2}\left(t,\alpha(t)\right)u_{2}(t)+b(t)\right]dt\\
   &\quad+\left[C\left(t,\alpha(t)\right)X(t)+D_{1}\left(t,\alpha(t)\right)u_{1}(t)+D_{2}\left(t,\alpha(t)\right)u_{2}(t)+\sigma(t)\right]dW(t),\qquad t\in[0,T],\\
   X(0)&=x,\quad \alpha(0)=i.
   \end{aligned}
   \right.
 \end{equation}
 Here, $X(\cdot)$ is controlled by two players and is called the \emph{state process} with an initial value $x\in\mathbb{R}^{n}$. For $k=1,2$, $u_{k}(\cdot)$ is called the \emph{control process} of Player $k$, which belongs to the Hilbert space $\mathcal{U}_{k}\triangleq L_{\mathbb{F}}^{2}(0,T;\mathbb{R}^{m_{k}})$. Any element $u_{k}(\cdot)\in\mathcal{U}_{k}$ is called an \emph{admissible control} of Player $k$.

 In this paper, we suppose that the coefficients of the state equation \eqref{state} satisfy the following assumption:\\
 \textbf{(A1)} For $i\in\mathcal{S}$ and $k=1,2$,
 $$A(\cdot,i), \text{ }C(\cdot,i)\in L^{\infty}(0,T;\mathbb{R}^{n\times n}),\quad B_{k}(\cdot,i),\text{ }D_{k}(\cdot,i)\in L^{\infty}( 0,T; \mathbb{R}^{n\times m_{k}}),\quad b(\cdot),\text{ }\sigma(\cdot)\in L_{\mathbb{F}}^{2}(0,T;\mathbb{R}^{n}).$$
 Obviously, the state equation \eqref{state} admits a unique solution $X(\cdot;x,i,u_{1},u_{2})\in\mathcal{S}_{\mathbb{F}}^{2}(0,T;\mathbb{R}^{n})$ for any initial value $(x,i)\in\mathbb{R}^{n}\times\mathcal{S}$ and the control pair $(u_{1}(\cdot),u_{2}(\cdot))\in\mathcal{U}_{1}\times\mathcal{U}_{2}$.

To measure the performance of $u_{1}(\cdot)$ and $u_{2}(\cdot)$, we introduce the following criterion functional:
\begin{equation}\label{performance-functional}
\begin{aligned}
   J(x,i;u_{1},u_{2})
    \triangleq &\mathbb{E}\left\{\displaystyle\int_{0}^{T}\left[
    \left<
    \left(
    \begin{matrix}
    Q(t,\alpha(t)) & S_{1}(t,\alpha(t))^{\top} & S_{2}(t,\alpha(t))^{\top} \\
    S_{1}(t,\alpha(t)) & R_{11}(t,\alpha(t)) & R_{12}(t,\alpha(t)) \\
    S_{2}(t,\alpha(t)) & R_{21}(t,\alpha(t)) & R_{22}(t,\alpha(t))
    \end{matrix}
    \right)
    \left(
    \begin{matrix}
    X(t) \\
    u_{1}(t) \\
    u_{2}(t)
    \end{matrix}
   \right),
    \left(
    \begin{matrix}
    X(t) \\
    u_{1}(t) \\
    u_{2}(t)
    \end{matrix}
    \right)
    \right>\right.\right.\\
    &\left.\left.+2\left<
    \left(
    \begin{matrix}
    q(t) \\
    \rho_{1}(t) \\
    \rho_{2}(t)
    \end{matrix}
   \right),
    \left(
    \begin{matrix}
    X(t) \\
    u_{1}(t) \\
    u_{2}(t)
    \end{matrix}
    \right)
   \right>
   \right]dt
   +\big<G(\alpha(T))X(T)+2g, X(T)\big>\right\}.
  \end{aligned}
\end{equation}
The above functional \eqref{performance-functional} can be regarded as the loss of Player $1$ and the gain of Player $2$. Therefore, Player 1 aims to minimize \eqref{performance-functional} by selecting a control $u_{1}$, while Player 2 aims to maximize \eqref{performance-functional} by selecting a control $u_{2}$. To ensure that the performance functional is well-defined, we assume that the weighting coefficients in \eqref{performance-functional} satisfy the following assumption:\\
\textbf{(A2)} For any fixed $i\in\mathcal{S}$ and $k=1,2$,
\begin{align*}
&Q(\cdot,i)\in L^{\infty}(0,T;\mathbb{S}^{n}),\quad S_{k}(\cdot,i)\in L^{\infty}(0,T;\mathbb{R}^{m_{k}\times n }),\quad R_{kk}(\cdot,i)\in L^{\infty}( 0,T; \mathbb{S}^{m_{k}}),\\
&R_{12}(\cdot,i)=R_{21}(\cdot,i)^{\top}\in L^{\infty}( 0,T; \mathbb{R}^{m_{1}\times m_{2}}),\quad G(i)\in \mathbb{S}^{n},\\
&q(\cdot)\in L_{\mathbb{F}}^{2}(0,T;\mathbb{R}^{n}),\quad
\rho_{k}(\cdot)\in L_{\mathbb{F}}^{2}(0,T;\mathbb{R}^{m_{k}}),\quad
 g\in L_{\mathcal{F}_{T}}^{2}(\Omega,\mathbb{R}^{n})
\end{align*}
The above-described problem is a two-person zero-sum SLQ differential game of the Markovian regime switching system, which can be summarized as follows:

\noindent\textbf{Problem (M-ZLQ):} For any $(x,i)\in \mathbb{R}^{n}\times \mathcal{S}$, find $(u_{1}^{*},u_{2}^{*})\in\mathcal{U}_{1}\times \mathcal{U}_{2}$ such that
\begin{equation}\label{ZLQ-value-function}
\sup_{u_{2}\in \mathcal{U}_{2}}J\left(x,i;u_{1}^{*},u_{2}\right)=J\left(x,i;u_{1}^{*},u_{2}^{*}\right)=\inf_{u_{1}\in \mathcal{U}_{1}}J\left(x,i;u_{1},u_{2}^{*}\right).\\
\end{equation}

If  $b=\sigma=q=0$, $\rho_{1}=0$, $\rho_{2}=0$, $g=0$, the corresponding state process, performance functional, and problem are denoted by  $X^{0}(\cdot;x,i,u_{1},u_{2})$, $J^{0}(x,i;u_{1},u_{2})$, Problem (M-ZLQ)$^0$,  respectively.

If we formally set $m=m_{1}+m_{2}$, $\mathcal{U}=\mathcal{U}_{1}\times \mathcal{U}_{2}=L_{\mathbb{F}}^{2}(0,T;\mathbb{R}^{m})$ and denote (for $i\in\mathcal{S}$)
\begin{equation}\label{natation-1}
\begin{aligned}
 &B(\cdot,i)\triangleq\left(B_{1}(\cdot,i),B_{2}(\cdot,i)\right),\quad  D(i)\triangleq\left(D_{1}(\cdot,i),D_{2}(\cdot,i)\right),\quad
 S(\cdot,i)^{\top}\triangleq\left( S_{1}(\cdot,i)^{\top}, S_{2}(\cdot,i)^{\top}\right),\\
 &R(\cdot,i)\triangleq\left(\begin{matrix} R_{1}(\cdot,i)\\R_{2}(\cdot,i)\end{matrix}\right)
  \triangleq\left(\begin{matrix} R_{11}(\cdot,i) & R_{12}(\cdot,i)\\R_{21}(\cdot,i)& R_{22}(\cdot,i)\end{matrix}\right)\quad
  \rho(\cdot)\triangleq\left(\begin{matrix} \rho_{1}(\cdot)\\\rho_{2}(\cdot)\end{matrix}\right)\quad
  u\triangleq\left(\begin{matrix} u_{1}(\cdot)\\u_{2}(\cdot)\end{matrix}\right),
\end{aligned}
\end{equation}
then corresponding state equation \eqref{state} and performance functional \eqref{performance-functional} can be rewritten as
\begin{equation}\label{state-notation}
   \left\{
   \begin{aligned}
   dX(t)&=\left[A\left(t,\alpha(t)\right)X(t)+B\left(t,\alpha(t)\right)u(t)+b(t)\right]dt\\
   &\quad+\left[C\left(t,\alpha(t)\right)X(t)+D\left(t,\alpha(t)\right)u(t)+\sigma(t)\right]dW(t),\qquad t\in[0,T],\\
   X(0)&=x,\quad \alpha(0)=i,
   \end{aligned}
   \right.
 \end{equation}
 and
 \begin{equation}\label{performance-functional-notation}
\begin{aligned}
   J(x,i;u)
    \triangleq &\mathbb{E}\left\{\int_{0}^{T}\left[
    \left<
    \left(
    \begin{matrix}
    Q(t,\alpha(t)) & S(t,\alpha(t))^{\top}  \\
    S(t,\alpha(t)) & R(t,\alpha(t))
    \end{matrix}
    \right)
    \left(
    \begin{matrix}
    X(t) \\
    u(t)
    \end{matrix}
   \right),
    \left(
    \begin{matrix}
    X(t) \\
    u(t)
    \end{matrix}
    \right)
    \right>\right.\right.\\
    &\left.\left.+2\left<
    \left(
    \begin{matrix}
    q(t) \\
    \rho(t)
    \end{matrix}
   \right),
    \left(
    \begin{matrix}
    X(t) \\
    u(t)
    \end{matrix}
    \right)
   \right>
   \right]dt
   +\big<G(\alpha(T))X(T)+2g, X(T)\big>\right\}.
  \end{aligned}
\end{equation}
If we further set $m_{1}=m$ and $m_{2}=0$, then the Problem (M-ZLQ) becomes an SLQ optimal control problem, which can be summarized as follows.

\noindent\textbf{Problem (M-SLQ):} For any given $(x,i)\in \mathbb{R}^{n}\times \mathcal{S}$, find a $u^{*}\in\mathcal{U}$ such that
\begin{equation}\label{SLQ-value-function}
J\left(x,i;u^{*}\right)=\inf_{u\in \mathcal{U}}J\left(x,i;u\right)=V(x,i).
\end{equation}
The function $V(\cdot,\cdot)$ is called the value function of the Problem (M-SLQ). If  $b=\sigma=q=0$, $\rho=0$, $g=0$, then the corresponding state process, performance functional, value function, and problem are denoted by  $X^{0}(\cdot;x,i,u)$, $J^{0}(x,i;u)$, $V^{0}(\cdot,\cdot)$, Problem (M-SLQ)$^0$,  respectively.

We note that the zero-sum SLQ differential game without Markovian jumps, i.e., $\mathcal{S}=\{1\}$, has been studied from various perspectives in \citet{mou2006two, Sun.J.R.2014_NILQ, yu2015optimal, Sun2021}. It is well known that the stochastic model involving regime-switching jumps has significant practical implications in fields such as engineering, financial management, and economics (see, for example, \cite{Zhang-Siu-Meng-2010,zhang-Elliott-Siu-Guo-2011,zhang_stochastic_2012,
zhang_general_2018,sun_risk-sensitive_2018}).  \citet{moon2019sufficient} investigated the Problem (M-ZLQ)$^{0}$ with $R_{12}(\cdot,\alpha(\cdot))=R_{21}(\cdot,\alpha(\cdot))^{\top}=0$ in 2019. However, he only considered the feedback saddle point for the Problem (M-ZLQ)$^{0}$ and did not analyze its open-loop solvability, closed-loop solvability, and the relationship between both. Additionally, the solvability of the corresponding CDREs in \cite{moon2019sufficient} is obtained only under the following condition:
\begin{equation}\label{Moon-condition}
  Q(\cdot,i), \text{ }G(i)\geq 0,\quad D_{k}(\cdot,i)=S_{k}(\cdot,i)=0,\quad R_{11}(\cdot,i)>0, \quad R_{22}(\cdot,i)=-\gamma^{2} I, \quad i\in\mathcal{S},\quad k=1,2,
\end{equation}
for a sufficiently large value of $\gamma>0$.
However, to our knowledge, the open-loop and closed-loop solvabilities for the Problem (M-ZLQ) have not yet been studied.  \citet{Zhang.X.2021_ILQM} examined the open-loop solvability and closed-loop solvability of Problem (M-SLQ). They established the equivalence between the open-loop (respectively, closed-loop) solvability and the existence of an adapted solution to the corresponding forward-backward stochastic differential equations with constraints (respectively, the existence of a regular solution to Riccati equations). However, whether those results can be generalized to Problem (M-ZLQ) remains unknown.

In this paper, we study the Problem (M-ZLQ) within a more general framework that allows $R_{12}(\cdot,\alpha(\cdot))=R_{21}(\cdot,\alpha(\cdot))^{\top}\neq 0$ and the presence of the inhomogeneous term. We discuss open-loop and closed-loop solvability by employing a new “decomposition method." Furthermore, we will examine the solvability of associated CDREs under the uniform convexity-concavity condition:\\
\textbf{Condition (UCC).} There exists a constant $\lambda>0$ such that
\begin{equation}\label{uniformly-convex-concave-condition}
  \left\{
  \begin{array}{lll}
  J^{0}(0,i;u_{1},0)\geq \lambda \mathbb{E}\int_{0}^{T}\big|u_{1}(t)\big|^{2}dt,&\forall u_{1}(\cdot)\in\mathcal{U}_{1},&\forall i\in\mathcal{S},\\[2mm]
  J^{0}(0,i;0,u_{2})\leq -\lambda \mathbb{E}\int_{0}^{T}\big|u_{2}(t)\big|^{2}dt,&\forall u_{2}(\cdot)\in\mathcal{U}_{2},&\forall i\in\mathcal{S}.
  \end{array}
  \right.
\end{equation}
Clearly, the condition \eqref{uniformly-convex-concave-condition} is much weaker than \eqref{Moon-condition} introduced in \cite{moon2019sufficient}.
As we can see in Theorem \ref{thm-ZLQ-open-loop}, the following weaker version of condition \eqref{uniformly-convex-concave-condition} is a necessary condition for the existence of open-loop saddle points for Problem (M-ZLQ): \\
\textbf{Condition (CC).} The performance functional \eqref{performance-functional} such that
\begin{equation}\label{convex-concave-condition}
  \left\{
  \begin{array}{lll}
  J^{0}(0,i;u_{1},0)\geq 0,&\forall u_{1}(\cdot)\in\mathcal{U}_{1},&\forall i\in\mathcal{S},\\[2mm]
  J^{0}(0,i;0,u_{2})\leq 0,&\forall u_{2}(\cdot)\in\mathcal{U}_{2},&\forall i\in\mathcal{S}.
  \end{array}
  \right.
\end{equation}
We point out that the condition (UCC) is only slightly stronger than the necessary condition (CC) in the sense that the condition (UCC) will degenerate into the condition (CC) by letting $\lambda\downarrow 0$.

\subsection{Brief history and contributions of this paper}
Differential games involving multi-person decision-making within a dynamic system are widely applied across various fields, including engineering, economics, biology, and reinsurance. The study of the deterministic linear quadratic differential game (DLQ-DG, for short) can be traced back to the pioneering work of \citet{ho1965differential}, which inspired further research in this area. \citet{schmitendorf1970existence} investigated the open-loop and closed-loop solvabilities for DLQ-DG problems and demonstrated that the existence of a closed-loop saddle point does not necessarily imply the existence of an open-loop saddle point. For recent developments, one can refer to \citet{bernhard1979linear,delfour2007linear,delfour2009linear,zhang2005some,bacsar2008h}, and the references cited therein. The study of two-person zero-sum stochastic differential games was first undertaken by \citet{fleming1989existence}, who proved that the upper and lower value functions satisfy the dynamic programming principle and are the unique viscosity solutions to the associated Hamilton–Jacobi–Bellman–Isaacs partial differential equations. Later, \citet{hamadene1995zero} discussed stochastic differential games based on the backward stochastic differential equation (BSDE) approach. Following this, several subsequent works emerged. For example, see \citet{buckdahn2008stochastic,wang2010pontryagin,wang2012partial,bayraktar2013weak,lv2020two} and others.

Many authors have regarded zero-sum SLQ differential games as a specific type of stochastic differential game. \citet{mou2006two} studied the open-loop solvability of the zero-sum SLQ differential games problem (denoted as Problem (ZLQ)) using the Hilbert space method. Consequently, \citet{Sun.J.R.2014_NILQ} successfully generalized the results in \cite{mou2006two} by further considering the closed-loop solvability of Problem (ZLQ). However, neither \citet{mou2006two} nor \citet{Sun.J.R.2014_NILQ} derived the solvability of the corresponding differential Riccati equation (DRE). They provided the equivalent characterization for the open-loop saddle points and the closed-loop equilibrium strategy. \citet{yu2015optimal} investigated the optimal control-strategy pair in a feedback form for Problem (ZLQ) based on the Riccati equation approach and studied the solvability of DRE under a special case (see the assumptions (H1.1)-(H1.4) and (H2.1)-(H2.2)).  Recently, \citet{Sun2021} revisited Problem (ZLQ) and provided the solvability result for DRE under the uniform convexity-concavity condition, which is more general than the assumptions (H1.1)-(H1.4) and (H2.1)-(H2.2) imposed in \citet{yu2015optimal}. There are many other works on SLQ differential games, among which we would like to mention the works \cite{Hamadene-1998-backward, Hamadene-1999-nonzero, wu2005forward, Yu-2012-LQG, Li-2015-Recursive, sun2019linear} on non-zero-sum SLQ differential games and the works \cite{bensoussan2016linear,tian2020closed} on mean-field SLQ differential games.

The main contributions of this paper can be summarized as follows.

(i) %We adopt a new “decomposition method" to obtain the equivalent characterizations of open-loop saddle points and closed-loop equilibrium strategy.
Both \citet{mou2006two} and \citet{Sun.J.R.2014_NILQ} regarded the SLQ control problem as a special case of Problem (ZLQ) and derived the open-loop and closed-loop solvabilities for the SLQ control problem based on the results in Problem (ZLQ). In this paper, we adopt a different approach to study the equivalent characterizations of open-loop saddle points and closed-loop equilibrium strategies for Problem (M-ZLQ). We first decompose the open-loop and closed-loop solvability problems of Problem (M-ZLQ) into two coupled SLQ control problems to solve. Then, based on the findings in \citet{Zhang.X.2021_ILQM}, we obtain the open-loop and closed-loop solvabilities for Problem (M-ZLQ). We note that this method is also applicable to non-zero-sum SLQ differential game problems, and the obtained results in this paper can be degenerated to the case without regime-switching jumps by setting $\mathcal{S}=\{1\}$.

(ii) The solvability of the associated CDREs \eqref{ZLQ-CCDREs} is derived under the uniform convexity-concavity condition \eqref{uniformly-convex-concave-condition}. Based on this result, we successfully constructed the open-loop saddle point and its closed-loop representation. It is worth mentioning that, unlike the single DRE studied in \citet{Sun2021}, the CDREs involved in this paper comprise a system of fully coupled nonlinear ordinary differential equations; thus, we cannot directly obtain their solvability using the technique provided in \citet{Sun2021}. With the aid of the dimension extension technique, we first employ permutation matrices to eliminate the coupling terms in the CDREs \eqref{ZLQ-CCDREs} (see \eqref{ZLQ-CCDREs-3}) and then establish the solvability of the CDREs  \eqref{ZLQ-CCDREs} based on the contraction mapping theorem and the continuation method.

(iii) At the end of this paper, we provide a concrete example for Problem (M-ZLQ) and rigorously verify that the provided example meets condition \eqref{uniformly-convex-concave-condition} by applying Gronwall’s inequality and exchanging the order of integration technique. Based on the derived results, we then obtain the closed-loop representation for open-loop saddle points at any initial value. Here, we assert that the method for verifying the condition \eqref{uniformly-convex-concave-condition} in this example is representative and can be similarly applied to other application scenarios.

The rest of the paper is organized as follows. Section \ref{section-2} presents some preliminary results that will be referenced frequently throughout this paper. Section \ref{section-3} aims to investigate the open-loop solvability of the Problem (M-ZLQ) and discuss the closed-loop representation for open-loop saddle points. Section \ref{section-4} studies closed-loop solvability and illustrates the relationship between open-loop solvability and closed-loop solvability.  Section  \ref{section-5} further explores the solvability of CDREs under the uniform convexity-concavity condition. Finally, Section  \ref{section-6} concludes the paper by providing concrete examples to illustrate the results developed in the previous sections.

\section{Preliminaries}\label{section-2}
Throughout this paper, we suppose that the Hilbert space $\mathbb{R}^{n\times m}$ is equipped with the Frobenius inner product $<M,N>\triangleq tr(M^{\top}N)$ for any $M,\, N\in\mathbb{R}^{n \times m}$, where $M^{\top}$ is the transpose of $M$ and $tr(M^{\top}N)$ is the trace of $M^{\top}N$. The norm introduced by the Frobenius inner product is denoted by $|\cdot|$. The range of matrix $M$  is denoted by $\mathcal{R}(M)$, and the identity matrix of size $n$ is denoted by $I_{n}$, which is often written as $I$ when no confusion occurs.
For any $M, N \in \mathbb{S}^n$, we write $M \geqslant N$ (respectively, $M>N$) if $M-N$ is semi-positive definite (respectively, positive definite). Specially, for an $\mathbb{S}^n$-valued measurable function $F(\cdot)$ on $[0, T]$, we write
$$\left\{
\begin{array}{lll}
F(\cdot) \geqslant 0 & \text { if } \quad F(s) \geqslant 0, & \text { a.e. } s \in[0, T], \\
F(\cdot)>0 & \text { if } \quad F(s)>0, & \text { a.e. } s \in[0, T], \\
F(\cdot) \gg 0 & \text { if } \quad F(s) \geqslant \delta I_n, & \text { a.e. } s \in[0, T], \quad \text { for some } \delta>0 .
\end{array}
\right.$$
Moreover, we denote $F(\cdot)\leq 0$, $F(\cdot)<0$ and $F(\cdot)\ll 0$ if $-F(\cdot)\geq 0$, $-F(\cdot)>0$ and $-F(\cdot)\gg 0$, respectively.

The generator of the Markov chain $\alpha(\cdot)$  under $\mathbb{P}$ is defined as $\pi(t):=\left[\pi_{ij}(t)\right]_{i,j=1,2,...,L}$. Here,  for $i\neq j$,  $\pi_{ij}(t)\geq 0$ is the bounded determinate transition intensity of the chain from state $i$ to state $j$ at time $t$ and $\sum_{j=1}^{L}\pi_{ij}(t)=0$ for any fixed $i$. In the following, let $N_{j}(t)$ be the number of jumps  into state $j$ up to time $t$ and set
$$\widetilde{N}_{j}(t)=N_{j}(t)-\int_{0}^{t}\sum_{i=1,i\neq j}^{L}\pi_{ij}(s)\mathbb{I}_{\{\alpha(s-)=i\}}ds,$$
where $\mathbb{I}_{A}$ is the indicator function.
Then for each $j\in\mathcal{S}$, the term $\widetilde{N}_{j}(t)$ is an $\left(\mathbb{F},\mathbb{P}\right)$-martingale (see, \citet{zhang_stochastic_2012}). For any given $L$-dimensional vector process $\mathbf{\Gamma}(\cdot)=\left[\Gamma_1(\cdot),\Gamma_2(\cdot),\cdots,\Gamma_L(\cdot)\right]$, we define
  $$\mathbf{\Gamma}(t)\cdot d\mathbf{\widetilde{N}}(t)\triangleq\sum_{j=1}^{L}\Gamma_{j}(s)d\widetilde{N}_{j}(s).$$
In addition, for any given Banach space $\mathbb{B}$, we denote
$$\mathcal{D}\left(\mathbb{B}\right)=\left\{\mathbf{\Lambda}(\cdot)=\left(\Lambda(\cdot,1),\cdots,\Lambda(\cdot,L)\right) \mid \Lambda(\cdot,i) \in \mathbb{B}\text{, } \forall i\in \mathcal{S}\right\}.$$
%Specifically, an element $\mathbf{\Theta}(\cdot)=\big(\Theta(\cdot,1),\Theta(\cdot,2),\cdots \Theta(\cdot,L)\big)\in \mathcal{D}\left(L^{2}(0,T;\mathbb{R}^{m\times n})\right)$ if and only if $\Theta(\cdot,i)\in L^{2}(0,T;\mathbb{R}^{m\times n})$ for any $i\in\mathcal{S}$.

Now, we recall some results about the pseudo-inverse introduced by \citet{Penrose.1955}.
\begin{lemma}\label{lem-Pseudoinverse}
\begin{description}
    \item[(i)] For any $M\in\mathbb{R}^{m\times n}$, there exists a unique matrix $M^{\dag
}\in\mathbb{R}^{n\times m}$ such that
$$MM^{\dag}M=M,\qquad M^{\dag}MM^{\dag}=M^{\dag},\qquad (MM^{\dag})^{\top}=MM^{\dag},\qquad
(M^{\dag}M)^{\top}=M^{\dag}M.$$
In addition, if $M\in\mathbb{S}^{n}$, then $M^{\dag}\in\mathbb{S}^{n}$, and
$$MM^{\dag}=M^{\dag}M,\qquad M\geq 0\Leftrightarrow M^{\dag}\geq 0.$$
In the above, $M^{\dag}$ is called the pseudo-inverse of $M$.
\item[(ii)] Let $L\in\mathbb{R}^{n\times k}$ and $N\in\mathbb{R}^{n\times m}$. The matrix equation $NX=L$ admits a solution if and only if $NN^{\dag}L=L$, in which case the general solution is given by $$X=N^{\dag}L+(I-N^{\dag}N)Y,$$
where $Y\in\mathbb{R}^{m\times k}$ is arbitrary.
  \end{description}
\end{lemma}

\begin{remark}\label{rmk-Pseudoinverse}
\begin{description}
\item[(i)] It follows from $MM^{\dag}M=M$ that
$$MM^{\dag}(I-MM^{\dag})=0,\quad M^{\dag}M(I-M^{\dag}M)=0,$$
which implies that both $MM^{\dag}$ and $M^{\dag}M$ are  orthogonal projection matrix.
    \item[(ii)] Obviously, the condition $NN^{\dag}L=L$ is equivalent to $\mathcal{R}(L)\subset \mathcal{R}(N)$.
    % Obviously, the condition $NN^{\dag}L=L$ implies that $\mathcal{R}(L)\subset \mathcal{R}(N)$. Conversely, supposing the condition $\mathcal{R}(L)\subset \mathcal{R}(N)$ holds, then we must have $L=NX$ for some matrix $X$. It follows from item (ii) of the above Lemma that $X=N^{\dag}L+[I-N^{\dag}N]Y$, which further implies that $L=NN^{\dag}L$.
\item[(iii)] In the above lemma, if $N\in\mathbb{S}^n$ and $NX=L$ admits a solution $X$, then
\begin{align*}
    X^{\top}NX&=\big[N^{\dag}L+(I-N^{\dag}N)Y\big]^{\top}N\big[N^{\dag}L+(I-N^{\dag}N)Y\big]=L^{\top}N^{\dag}L.
\end{align*}
%$X^{\top}NX=L^{\top}N^{\dag}L$.
  \end{description}
\end{remark}

The following Lemma can be found in \citet{Sun2021}, whose proof is straightforward.
\begin{lemma}\label{lem-matrix}
For $\mathbb{M}\in\mathbb{S}^{m}$, $\mathbb{L}\in\mathbb{R}^{m\times n}$, $\mathbb{N}\in\mathbb{S}^{n}$, if $\mathbb{M}$ and $\Phi\triangleq \mathbb{N}-\mathbb{L}^{\top}\mathbb{M}^{-1}\mathbb{L}$ are invertible, then $\left(\begin{matrix}
\mathbb{M} & \mathbb{L}\\ \mathbb{L}^{\top} & \mathbb{N}\end{matrix}\right)$ is also invertible and
$$
\left(\begin{matrix}
\mathbb{M} & \mathbb{L}\\ \mathbb{L}^{\top} & \mathbb{N}\end{matrix}\right)^{-1}
=\left(\begin{matrix}
\mathbb{M}^{-1}+(\mathbb{M}^{-1}\mathbb{L})\Phi^{-1}(\mathbb{M}^{-1}\mathbb{L})^{\top} & -(\mathbb{M}^{-1}\mathbb{L})\Phi^{-1}\\
-\Phi^{-1}(\mathbb{M}^{-1}\mathbb{L})^{\top} & \Phi^{-1}
\end{matrix}\right).
$$
Moreover, for every $\theta\in\mathbb{R}^{m\times k}$ and $\xi\in\mathbb{R}^{n\times k}$,
$$
\left(\theta^{\top},\xi^{\top}\right)\left(\begin{matrix}
\mathbb{M} & \mathbb{L}\\ \mathbb{L}^{\top} & \mathbb{N}\end{matrix}\right)^{-1}\left(\begin{matrix}\theta\\ \xi\end{matrix}\right)
=\theta^{\top}\mathbb{M}\theta+(\mathbb{L}^{\top}\mathbb{M}^{-1}\theta-\xi)^{\top}\Phi^{-1}
(\mathbb{L}^{\top}\mathbb{M}^{-1}\theta-\xi).
$$
In particular, if $\mathbb{M}$ is positive definite and $\mathbb{N}$ is negative definite, then
$$
\left(\theta^{\top},\xi^{\top}\right)\left(\begin{matrix}
\mathbb{M} & \mathbb{L}\\ \mathbb{L}^{\top} & \mathbb{N}\end{matrix}\right)^{-1}\left(\begin{matrix}\theta\\ \xi\end{matrix}\right)
\leq \theta^{\top}\mathbb{M}\theta,\quad \forall \theta\in\mathbb{R}^{m\times k},\text{ }\xi\in\mathbb{R}^{n\times k}.
$$
\end{lemma}

% \begin{lemma}[Extended Schur's lemma \cite{Albert.1969}]\label{lem-Schur}
% Let $M\in\mathbb{S}^{n}$, $N\in\mathbb{S}^{m}$, $L\in\mathbb{R}^{n\times m}$. Then the following conditions are equivalent:
% \begin{description}
%   \item[(i)] $M-LN^{\dag}L^{\top}\geq 0$, $N\geq 0$, and $L(I-NN^{\dag})=0$.
%   \item[(ii)] $\left(\begin{array}{cc}
%       M &  L\\
%       L^{\top}& N
%   \end{array}\right)\geq 0$.
% \end{description}
% \end{lemma}

At the end of this section, we present some results regarding the Problem (M-SLQ), which was introduced in the previous section. All proofs of those results can be found in  \citet{Zhang.X.2021_ILQM}, and we omit them here.

If we replace the object \eqref{SLQ-value-function} of Problem (M-SLQ) to that finding a $u^{*}(\cdot)\in \mathcal{U}$ to maximize the $J\left(x,i;u\right)$, then we can solve the corresponding control problem by setting
  \begin{equation}\label{SLQ-relation}
  \bar{J}\left(x,i;u\right)=-J\left(x,i;u\right),\quad \forall (x,i,u(\cdot))\in\mathbb{R}^{n}\times\mathcal{S}\times  \mathcal{U}.
  \end{equation}
 In the following, we define the control problem of finding a $u^{*}(\cdot)\in \mathcal{U}$ to maximize the $J\left(x,i;u\right)$ as Problem ($\overline{\text{M-SLQ}}$) and denote the corresponding homogeneous control problem as Problem ($\overline{\text{M-SLQ}}$)$^{0}$.

For given $(\mathbf{\Theta}(\cdot),\nu(\cdot))\in \mathcal{D}\left(L^{2}(0,T;\mathbb{R}^{m\times n})\right)\times L_{\mathbb{F}}^{2}(0,T;\mathbb{R}^{m})$, let
\begin{equation}\label{closed-outcome}
u(\cdot;x,i)\triangleq \Theta(\cdot,\alpha(\cdot))X(\cdot;x,i,\mathbf{\Theta},\nu)+\nu(\cdot)\in L_{\mathbb{F}}^{2}(0,T;\mathbb{R}^{m}),
\end{equation}
where $X(\cdot;x,i,\mathbf{\Theta},\nu)$ is the solution of following SDE:
\begin{equation}\label{state-LQ-closed}
   \left\{
   \begin{aligned}
   dX(t)&=\left[\big(A(t,\alpha(t))+B(t,\alpha(t))\Theta(t,\alpha(t))\big)X(t)+B(t,\alpha(t))\nu(t)+b(t)\right]dt\\
   &\quad+\left[\big(C(t,\alpha(t))+D(t,\alpha(t))\Theta(t,\alpha(t))\big)X(t)+D(t,\alpha(t))\nu(t)+\sigma(t)\right]dW(t),\qquad t\in[0,T],\\
 X(0)&=x,\quad \alpha(0)=i.
   \end{aligned}
   \right.
 \end{equation}
 In the above, we denote $(\mathbf{\Theta}(\cdot),\nu(\cdot))$ as the closed-loop strategy of Problem (M-SLQ) and define the $u(\cdot;x,i)$ in \eqref{closed-outcome} as the outcome control of closed-loop strategy $(\mathbf{\Theta}(\cdot),\nu(\cdot))$ for initial value $(x,i)$. Additionally, we denote
 \begin{equation}\label{SLQ-cost-closed}
J(x,i;\mathbf{\Theta},\nu)\triangleq J(x,i;u)\quad \text{with} \quad u(\cdot)\equiv u(\cdot;x,i) \text{ defined in \eqref{closed-outcome}}.
 \end{equation}

Now, we recall the definition of open-loop optimal control and closed-loop optimal strategy introduced in \citet{Zhang.X.2021_ILQM}.
\begin{definition}[Open-loop] An element $u^{*}(\cdot)\in L_{\mathbb{F}}^{2}(0,T;\mathbb{R}^{m})$ is called an open-loop optimal control of Problem (M-SLQ) (respectively, Problem ($\overline{\text{M-SLQ}}$)) for the initial value $(x,i)\in \mathbb{R}^{n}\times \mathcal{S}$ if
 $$J(x,i;u^{*})\leq J(x,i;u)\text{ }  (\text{respectively, } J(x,i;u^{*})\geq J(x,i;u)),\quad \forall u(\cdot)\in L_{\mathbb{F}}^{2}(0,T;\mathbb{R}^{m}).$$
\end{definition}

\begin{definition}[Closed-loop]
A pair $(\mathbf{\Theta}^{*}(\cdot),\nu^{*}(\cdot))\in \mathcal{D}\left(L^{2}(0,T;\mathbb{R}^{m\times n})\right)\times L_{\mathbb{F}}^{2}(0,T;\mathbb{R}^{m})$ is called a closed-loop optimal strategy of Problem (M-SLQ) (respectively, Problem ($\overline{\text{M-SLQ}}$))  if for any $(x,i)\in\mathbb{R}^{n}\times\mathcal{S}$,
 $$J(x,i;\mathbf{\Theta}^{*},\nu^{*})\leq J(x,i;u)\text{ } (\text{respectively, } J(x,i;\mathbf{\Theta}^{*},\nu^{*})\geq J(x,i;u)),\quad \forall u(\cdot)\in L_{\mathbb{F}}^{2}(0,T;\mathbb{R}^{m}).$$
\end{definition}

The following result provides an equivalent characterization for open-loop optimal control in terms of FBSDEs.
\begin{lemma}\label{thm-SLQ-open}
  An element $u^{*}(\cdot)\in L_{\mathbb{F}}^{2}(0,T;\mathbb{R}^{m})$ is an open-loop optimal control of Problem (M-SLQ) (respectively, Problem ($\overline{\text{M-SLQ}}$)) for the initial value $(x,i)\in \mathbb{R}^{n}\times \mathcal{S}$ if and only if
  \begin{description}
    \item[(i)] the following convexity (respectively, concavity) condition holds:
    \begin{equation}\label{SLQ-convex-condition}
      J^{0}(0,i;u)\geq 0\text{ } (\text{respectively, } J^{0}(0,i;u)\leq 0),\quad \forall u(\cdot)\in L_{\mathbb{F}}^{2}(0,T;\mathbb{R}^{m}), \quad \forall i\in \mathcal{S};
    \end{equation}
    \item[(ii)] the adapted solution $\left(X^{*},Y^{*},Z^{*},\mathbf{\Gamma}^{*}\right)\in \mathcal{S}_{\mathbb{F}}^{2}(0,T;\mathbb{R}^{n})\times \mathcal{S}_{\mathbb{F}}^{2}(0,T;\mathbb{R}^{n})\times L_{\mathbb{F}}^{2}(0,T;\mathbb{R}^{n})\times\mathcal{D}\left(L_{\mathcal{P}}^{2}(0,T;\mathbb{R}^{n})\right)$ to the  FBSDEs
    \begin{equation}\label{SLQ-FBSDE}
      \left\{
      \begin{aligned}
      dX^{*}(t)&=\left[A(t,\alpha(t))X^{*}(t)+B(t,\alpha(t))u^{*}(t)+b(t)\right]dt\\
   &\quad+\left[C(t,\alpha(t))X^{*}(t)+D(t,\alpha(t))u^{*}(t)+\sigma(t)\right]dW(t),\\
      dY^{*}(t)&=-\big[A(t,\alpha(t))^{\top}Y^{*}(t)+C(t,\alpha(t))^{\top}Z^{*}(t)
      +Q(t,\alpha(t))^{\top}X^{*}(t)\\
      &\quad+S(t,\alpha(t))^{\top}u^{*}(t)+q(t)\big]dt+Z^{*}(t)dW(t)+\mathbf{\Gamma}^{*}(t)\cdot d\widetilde{\mathbf{N}}(t),\quad t\in[0,T],\\
      X^{*}(0)&=x,\quad \alpha(0)=i,\quad Y^{*}(T)=G(\alpha(T)) X^{*}(T)+g,
      \end{aligned}
      \right.
    \end{equation}
    satisfies the following stationary condition:
  \begin{equation}\label{SLQ-stationary-condition}
     B(t,\alpha(t))^{\top}Y^{*}(t)+ D(t,\alpha(t))^{\top}Z^{*}(t)+S(t,\alpha(t))X^{*}(t)
     +R(t,\alpha(t))u^{*}(t)+\rho(t)=0,\quad a.e.\quad a.s..
  \end{equation}
  \end{description}
\end{lemma}

For given $\mathbf{P}(\cdot)\equiv\left(P(\cdot,1),P(\cdot,2),\cdots,P(\cdot,L)\right)\in \mathcal{D}\left(C\big([0,T],\mathbb{S}^{n}\big)\right)$, we introduce the following notations:
\begin{equation}\label{notation-LMN}
\left\{
\begin{array}{l}
\mathcal{M}(t;P,i)\triangleq P(t,i)A(t,i)+A(t,i)^{\top}P(t,i)+C(t,i)^{\top}P(t,i)C(t,i)+Q(t,i)+\sum_{j=1}^{L}\pi_{ij}(t)P(t,j),\\[2mm]
\mathcal{L}(t;P,i)\triangleq P(t,i)B(t,i)+C(t,i)^{\top}P(t,i)D(t,i)+S(t,i)^{\top},\\[2mm]
\mathcal{N}(t;P,i)\triangleq D(t,i)^{\top}P(t,i)D(t,i)+R(t,i),\quad t\in[0,T],\quad i\in\mathcal{S}.
\end{array}
\right.
\end{equation}
   Then, the following result provides the closed-loop solvabilities of Problem (M-SLQ) and Problem ($\overline{\text{M-SLQ}}$) in terms of CDREs.

\begin{lemma}\label{thm-SLQ-closed}
  Problem (M-SLQ) (respectively, Problem ($\overline{\text{M-SLQ}}$)) is closed-loop solvable if and only if the following conditions hold:
   \begin{description}
   \item[(i)] The CDREs
   \begin{equation}\label{CDREs}
     \left\{
     \begin{aligned}
     &\dot{P}(t,i)+\mathcal{M}(t;P,i)-\mathcal{L}(t;P,i)\mathcal{N}(t;P,i)^{\dag}\mathcal{L}(t;P,i)^{\top}=0,\quad t\in[0,T],\\
     &P(T,i)=G(i),\quad i\in\mathcal{S},
     \end{aligned}
     \right.
   \end{equation}
   admits a unique solution $\mathbf{P}(\cdot)\in \mathcal{D}\left(C\big([0,T],\mathbb{S}^{n}\big)\right)$ such that
   \begin{equation}\label{CDREs-constraint}
       \left\{
       \begin{aligned}
       &\mathcal{R}\left(\mathcal{L}(t,P,i)^{\top}\right)\subseteq \mathcal{R}\left(\mathcal{N}(t,P,i)\right),\quad a.e.\text{ }t\in[0,T],\\
       &\mathcal{N}(\cdot,P,i)^{\dag}\mathcal{L}(\cdot,P,i)^{\top}\in L^{2}(0,T;\mathbb{R}^{m\times n}),\\
       &\mathcal{N}(t,P,i)\geq 0 \text{ } (\text{respectively, }  \mathcal{N}(t,P,i)\leq 0),\quad a.e.\text{ }t\in[0,T];
       \end{aligned}
       \right.
   \end{equation}
   \item[(ii)] The adapted solution $\left(\eta(\cdot),\zeta(\cdot),\mathbf{z}(\cdot)\right)$ of the following BSDE
   \begin{equation}\label{eta}
       \left\{
      \begin{aligned}
        d\eta(t)&=-\big\{\big[A(\alpha)^{\top}-\mathcal{L}(P,\alpha)\mathcal{N}(P,\alpha)^{\dag}B(\alpha)^{\top}\big]\eta+\big[C(\alpha)^{\top}-\mathcal{L}(P,\alpha)\mathcal{N}(P,\alpha)^{\dag}D(\alpha)^{\top}\big]\zeta\\
        &\quad+\big[C(\alpha)^{\top}-\mathcal{L}(P,\alpha)\mathcal{N}(P,\alpha)^{\dag}D(\alpha)^{\top}\big]P(\alpha)\sigma-\mathcal{L}(P,\alpha)\mathcal{N}(P,\alpha)^{\dag}\rho+P(\alpha)b+q\big\}dt\\
        &\quad+\zeta dW(t)+\mathbf{z}\cdot d\mathbf{\widetilde{N}}(t),\quad t\in [0,T],\\
        \eta(T)&=g,
      \end{aligned}
      \right.
   \end{equation}
   satisfies
   \begin{equation}\label{eta-constraint}
       \left\{
       \begin{aligned}
       &\widetilde{\rho}(t)\in \mathcal{R}\left(\mathcal{N}(t,P,i)\right),\quad a.e.\text{ }t\in[0,T],\\
       & \mathcal{N}(\cdot,P,\alpha(\cdot))^{\dag}\widetilde{\rho}(\cdot)\in L_{\mathbb{F}}^{2}(0,T;\mathbb{R}^{m}),
       \end{aligned}
       \right.
   \end{equation}
with
\begin{equation}\label{SLQ-rho}
    \widetilde{\rho}(t)\triangleq B(t,\alpha(t))^{\top}\eta(t)+D(t,\alpha(t))^{\top}\zeta(t)+D(t,\alpha(t))^{\top}P(t,\alpha(t))\sigma(t)+\rho(t),\quad a.e.\text{ }t\in[0,T],
\end{equation}
  \end{description}
  In this case, the closed-loop optimal strategy $(\widehat{\mathbf{\Theta}}(\cdot),\widehat{\nu}(\cdot))\in \mathcal{D}\left(L^{2}(0,T;\mathbb{R}^{m\times n})\right)\times L_{\mathbb{F}}^{2}(0,T;\mathbb{R}^{m})$ of Problem (M-SLQ) (respectively, Problem ($\overline{\text{M-SLQ}}$))  admits the following representation:
  \begin{equation}\label{SLQ-closed-optimal}
      \left\{
      \begin{aligned}
      &\widehat{\Theta}(\cdot,i)=-\mathcal{N}(\cdot,P,i)^{\dag}\mathcal{L}(\cdot,P,i)^{\top}
      +\big[I-\mathcal{N}(\cdot,P,i)^{\dag}\mathcal{N}(\cdot,P,i)\big]\Pi(\cdot,i),\quad i\in\mathcal{S},\\
      &\widehat{\nu}(\cdot)=-\mathcal{N}(\cdot,P,\alpha(\cdot))^{\dag}\widetilde{\rho}(\cdot)
      +\big[I-\mathcal{N}(\cdot,P,\alpha(\cdot))^{\dag}\mathcal{N}(\cdot,P,\alpha(\cdot))\big]\nu(\cdot),
      \end{aligned}
      \right.
  \end{equation}
  for some $\Pi(\cdot,i)\in L^{2}(0,T;\mathbb{R}^{m\times n})$ and $\nu(\cdot)\in L_{\mathbb{F}}^{2}(0,T;\mathbb{R}^{m})$, and the value function is given by
  \begin{equation}\label{SLQ-value}
      V(x,i)=\mathbb{E}\Big\{\big<P(0,i)x,x\big>+2\big<\eta(0),x\big>+\int_{0}^{T}\big[
      \big<P(\alpha)\sigma,\sigma\big>+2\big<\eta,b\big>+2\big<\zeta,\sigma\big>
      -\big<\mathcal{N}(P,\alpha)^{\dag}\widetilde{\rho},\widetilde{\rho}\big>\big]dt\Big\}.
  \end{equation}
\end{lemma}

In the above, we omit the argument $t$ for simplicity. If no confusion arises, such abbreviated notations will be frequently used in the rest of the paper.

\begin{remark}\label{rmk-SLQ-closed-2}\rm
Clearly, according to the basic properties of the pseudo-inverse, one can easily verify that
a pair $(\widehat{\mathbf{\Theta}}(\cdot),\widehat{\nu}(\cdot))\in \mathcal{D}\left(L^{2}(0,T;\mathbb{R}^{m\times n})\right)\times L_{\mathbb{F}}^{2}(0,T;\mathbb{R}^{m})$ is a closed-loop optimal strategy of Problem (M-SLQ) (respectively, Problem ($\overline{\text{M-SLQ}}$) ) if and only if the following conditions hold:
\begin{description}
\item[(i)] The CDREs \eqref{CDREs} admits a unique solution such that
 \begin{equation}\label{CDREs-constraint-2}
 \left\{
 \begin{aligned}
       &\mathcal{N}(t,P,i)\widehat{\Theta}(t,i)+\mathcal{L}(t,P,i)^{\top}=0,
       \quad a.e.\text{ }t\in[0,T],\\
       &\mathcal{N}(t,P,i)\geq 0\text{ } (\text{respectively, } \mathcal{N}(t,P,i)\leq 0),\quad a.e.\text{ }t\in[0,T];
\end{aligned}
\right.
   \end{equation}
 \item[(ii)]  The solution $\left(\eta(\cdot),\zeta(\cdot),\mathbf{z}(\cdot)\right)$ of the  BSDE \eqref{eta} satisfies
 \begin{equation}\label{eta-constraint-2}
      \mathcal{N}(t,P,\alpha(t))\widehat{\nu}(t)+\widetilde{\rho}(t)=0,\quad a.e.\text{ } t\in[0,T].
   \end{equation}
\end{description}
\end{remark}

The following result further provides the solvability of CDREs \eqref{CDREs}.
\begin{lemma}\label{thm-SLQ-uniform-convex}
  If for some constant $\lambda>0$,
  $$J^{0}(x,i;u)\geq \lambda\mathbb{E}\int_{0}^{T}\big|u(t)\big|^2dt\text{ } (\text{respectively, } J^{0}(x,i;u)\leq -\lambda\mathbb{E}\int_{0}^{T}\big|u(t)\big|^2dt),$$
  then the CDREs \eqref{CDREs} admits a unique solution such that
  $\mathcal{N}(\cdot,P,i)\gg0$ (respectively, $\mathcal{N}(\cdot,P,i)\ll 0$) for all $ i\in\mathcal{S}.$
    In particular, if the standard condition
  $$
  \begin{array}{c}
  G(i)\geq 0,\quad R(\cdot,i)\gg0,\quad Q(\cdot,i)-S(\cdot,i)^{\top}R(\cdot,i)S(\cdot,i)\geq 0,\quad i\in\mathcal{S},\\[2mm]
  (\text{respectively, }G(i)\leq 0,\quad R(\cdot,i)\ll 0,\quad Q(\cdot,i)-S(\cdot,i)^{\top}R(\cdot,i)S(\cdot,i)\leq 0,\quad i\in\mathcal{S}),
  \end{array}
  $$
  holds, then the CDREs \eqref{CDREs} admits a unique solution such that
  $P(\cdot,i)\geq 0$ (respectively, $P(\cdot,i)\leq 0$) for all $ i\in\mathcal{S}.$
\end{lemma}

\section{Open-loop solvability}\label{section-3}
This section discusses the open-loop solvability of the Problem (M-ZLQ). For simplicity, we assume that notations defined in \eqref{natation-1} are always valid in the rest of the paper.
Now, we introduce the definition of open-loop solvability for Problem (M-ZLQ).
\begin{definition}\label{def-ZLQ-open-loop-equilibrium}
A pair $(u_{1}^{*},u_{2}^{*})\in\mathcal{U}_{1}\times \mathcal{U}_{2}$ is called an open-loop saddle point of Problem (M-ZLQ) for the initial value $(x,i)\in\mathbb{R}^{n}\times\mathcal{S}$ if
\begin{equation}\label{ZLQ-open-loop-Nash-equilibrium-point}
    \begin{aligned}
    J\left(x,i;u_{1}^{*},u_{2}\right)\leq J\left(x,i;u_{1}^{*},u_{2}^{*}\right)\leq J\left(x,i;u_{1},u_{2}^{*}\right),\quad \forall \left(u_{1},u_{2}\right)\in \mathcal{U}_{1}\times \mathcal{U}_{2}.
    \end{aligned}
\end{equation}
Problem (M-ZLQ) is said to be (uniquely) open-loop solvable at $(x,i)$ if it admits a (unique)  open-loop saddle point $(u_{1}^{*},u_{2}^{*})\in \mathcal{U}_{1}\times \mathcal{U}_{2}$. In addition, Problem (M-ZLQ) is said to be (uniquely) open-loop solvable if it is  (uniquely) open-loop solvable at all $(x,i)\in\mathbb{R}^{n}\times\mathcal{S}$.
\end{definition}

The following result characterizes the open-loop solvability of the Problem (M-ZLQ) in terms of FBSDEs.

\begin{theorem}\label{thm-ZLQ-open-loop}
Let assumptions (A1)-(A2) hold. Then $u^{*}(\cdot)\equiv(u_{1}^{*}(\cdot)^{\top},u_{2}^{*}(\cdot)^{\top})^{\top}\in\mathcal{U}$ is an open-loop saddle point of Problem (M-ZLQ) for initial value $(x,i)\in\mathbb{R}^{n}\times\mathcal{S}$ if and only if:
 \begin{description}
    \item[(i)] The convexity-concavity condition \eqref{convex-concave-condition} holds;
 \item[(ii)] The adapted solution $\left(X^{*},Y^{*},Z^{*},\mathbf{\Gamma}^{*}\right)\in \mathcal{S}_{\mathbb{F}}^{2}(0,T;\mathbb{R}^{n})\times \mathcal{S}_{\mathbb{F}}^{2}(0,T;\mathbb{R}^{n})\times L_{\mathbb{F}}^{2}(0,T;\mathbb{R}^{n})\times\mathcal{D}\left(L_{\mathcal{P}}^{2}(0,T;\mathbb{R}^{n})\right)$ to the  FBSDEs \eqref{SLQ-FBSDE} satisfies the  stationary condition \eqref{SLQ-stationary-condition}.
   %\item[(ii)] The adapted solution $\left(X^{*},Y^{*},Z^{*},\mathbf{\Gamma}^{*}\right)\in \mathcal{S}_{\mathbb{F}}^{2}(\mathbb{R}^{n})\times \mathcal{S}_{\mathbb{F}}^{2}(\mathbb{R}^{n})\times L_{\mathbb{F}}^{2}(\mathbb{R}^{n})\times\mathcal{D}\left(L_{\mathcal{P}}^{2}(\mathbb{R}^{n})\right)$ to the FBSDE
%  \begin{equation}\label{ZLQ-FBSDEs}
%  \left\{
%      \begin{aligned}
%      dX^{*}(t)&=\left[A(t,\alpha(t))X^{*}(t)+B(t,\alpha(t))u^{*}(t)+b(t)\right]dt\\
%      &\quad+\left[C(t,\alpha(t))X^{*}(t)+D(t,\alpha(t))u^{*}(t)+\sigma(t)\right]dW(t),\\
%      dY^{*}(t)&=-\big[A(t,\alpha(t))^{\top}Y^{*}(t)+C(t,\alpha(t))^{\top}Z^{*}(t)+Q(t,\alpha(t))X^{*}(t)\\
%      &\quad+S(t,\alpha(t))^{\top}u^{*}(t)+q(t)\big]dt+Z^{*}(t)dW(t)+\mathbf{\Gamma}^{*}(t)\cdot d\mathbf{\widetilde{N}}(t),\quad t\in [0,T],\\
%      X^{*}(0)&=x,\quad\alpha_{0}=i, \quad Y^{*}(T)=M(\alpha(T))X^{*}(T)+m,
%      \end{aligned}
%      \right.
%  \end{equation}
%  satisfies the following stationary condition:
%  \begin{equation}\label{ZLQ-stationary-condition}
%     B(t,\alpha(t))^{\top}Y^{*}(t)+ D(t,\alpha(t))^{\top}Z^{*}(t)+S(t,\alpha(t))X^{*}(t)+R(t,\alpha(t))u^{*}(t)+\rho(t)=0,\quad a.e.\quad a.s..
%  \end{equation}
 \end{description}
\end{theorem}

\begin{proof}
In the rest of this paper, we denote the solution to \eqref{state} as $X(\cdot; x,i,u_{1},u_{2})$ and as $X^{0}(\cdot; x,i,u_{1},u_{2})$ if $b(\cdot)=\sigma(\cdot)=0$.
  For  given $(u_{1}^{*},u_{2}^{*})\in\mathcal{U}_{1}\times \mathcal{U}_{2}$, let $X_{1}\equiv X(\cdot;x,i,u_{1},u_{2}^{*})$, $X_{2}\equiv X(\cdot;x,i,u_{1}^{*},u_{2})$,
\begin{align*}
\bar{J}_{1}(x,i;u_{1})& \triangleq \mathbb{E}\Big\{\int_{0}^{T}\left[\left<\left(\begin{matrix}
Q(\alpha) & S_{1}(\alpha)^{\top}\\
S_{1}(\alpha) & R_{11}(\alpha)\end{matrix}\right)\left(\begin{matrix}
X_{1}\\u_{1}\end{matrix}\right),
\left(\begin{matrix}
X_{1}\\u_{1}\end{matrix}\right)\right>
+2\left<\left(\begin{matrix}
q+S_{2}(\alpha)^{\top}u_{2}^{*}\\
\rho_{1}+R_{12}(\alpha)u_{2}^{*}\end{matrix}\right),
\left(\begin{matrix}
X_{1}\\u_{1}\end{matrix}\right)\right>\right]dt\\
&\quad +\big<G(T,\alpha(T))X_{1}(T)+2g, X_{1}(T)\big>\Big\},
  \end{align*}
and
\begin{align*}
\bar{J}_{2}(x,i;u_{1})& \triangleq \mathbb{E}\Big\{\int_{0}^{T}\left[\left<\left(\begin{matrix}
Q(\alpha) & S_{2}(\alpha)^{\top}\\
S_{2}(\alpha) & R_{22}(\alpha)\end{matrix}\right)\left(\begin{matrix}
X_{2}\\u_{2}\end{matrix}\right),
\left(\begin{matrix}
X_{2}\\u_{2}\end{matrix}\right)\right>
+2\left<\left(\begin{matrix}
q+S_{1}(\alpha)^{\top}u_{1}^{*}\\
\rho_{2}+R_{21}(\alpha)u_{1}^{*}\end{matrix}\right),
\left(\begin{matrix}
X_{2}\\u_{2}\end{matrix}\right)\right>\right]dt\\
&\quad +\big<G(T,\alpha(T))X_{2}(T)+2g, X_{2}(T)\big>\Big\}.
  \end{align*}
  Then
  \begin{align*}
& J(x,i;u_{1},u_{2}^{*}) =\mathbb{E}\int_{0}^{T}\big<R_{22}(\alpha)u_{2}^{*}+2\rho_{2},u_{2}^{*}\big>dt +\bar{J}_{1}(x,i;u_{1}),\\
& J(x,i;u_{1}^{*},u_{2}) =\mathbb{E}\int_{0}^{T}\big<R_{11}(\alpha)u_{1}^{*}+2\rho_{1},u_{1}^{*}\big>dt +\bar{J}_{2}(x,i;u_{2}).
  \end{align*}
Now, we define the following two coupled SLQ control problems.\\
\noindent\textbf{Problem ($\overline{\text{M-SLQ}}$)$_{1}$:} For any given $(x,i)\in \mathbb{R}^{n}\times \mathcal{S}$, find a $u_{1}^{*}\in \mathcal{U}_{1}$ such that
\begin{equation}\label{J-1}
\bar{J}_{1}\left(x,i;u_{k}^{*}\right)=\inf_{u_{1}\in \mathcal{U}_{1}}\bar{J}_{1}\left(x,i;u_{1}\right).
\end{equation}
\noindent\textbf{Problem ($\overline{\text{M-SLQ}}$)$_{2}$:} For any given $(x,i)\in \mathbb{R}^{n}\times \mathcal{S}$, find a $u_{2}^{*}\in \mathcal{U}_{2}$ such that
\begin{equation}\label{J-2}
\bar{J}_{2}\left(x,i;u_{2}^{*}\right)=\sup_{u_{2}\in \mathcal{U}_{2}}\bar{J}_{2}\left(x,i;u_{2}\right).
\end{equation}
Any $u_{1}^{*}(\cdot)\in\mathcal{U}_{1}$ (respectively, $u_{2}^{*}(\cdot)\in\mathcal{U}_{2}$) satisfying the equation \eqref{J-1} (respectively, equation \eqref{J-2})
 is called an open-loop optimal control for Problem  ($\overline{\text{M-SLQ}}$)$_{1}$ (respectively,  ($\overline{\text{M-SLQ}}$)$_{2}$).

For any initial value $(x,i)\in\mathbb{R}^{n}\times\mathcal{S}$,  a pair $(u_{1}^{*},u_{2}^{*})$ is an open-loop saddle point of Problem (M-ZLQ) if and only if $u_{k}^{*}$ is an open-loop optimal control of Problem  ($\overline{\text{M-SLQ}}$)$_{k}$, which is further equivalent to the following conditions by Lemma \ref{thm-SLQ-open} and notation \eqref{natation-1}:
%It is clear that $(u_{1}^{*},u_{2}^{*})$ is an open-loop saddle point of Problem (M-ZLQ) for initial value $(x,i)\in\mathbb{R}^{n}\times\mathcal{S}$ if and only if $u_{k}^{*}$ is an open-loop optimal control of Problem  ($\overline{\text{M-SLQ}}$)$_{k}$ for initial value $(x,i)\in\mathbb{R}^{n}\times\mathcal{S}$, $k=1,2$. Consequently, by Lemma \ref{thm-SLQ-open} and notation \eqref{natation-1}, we  obtain that a pair $(u_{1}^{*},u_{2}^{*})$ is an open-loop saddle point of Problem (M-ZLQ) for initial value $(x,i)\in\mathbb{R}^{n}\times\mathcal{S}$ if and only if
\begin{description}
  \item[(i)] the following convexity-concavity condition holds
  \begin{equation}\label{convex-concave}
    \left\{
    \begin{aligned}
      &\bar{J}_{1}^{0}(0,i;u_{1})\geq 0,\quad \forall u_{1}(\cdot)\in\mathcal{U}_{1},\quad i\in\mathcal{S},\\
      &\bar{J}_{2}^{0}(0,i;u_{2})\leq 0,\quad \forall u_{2}(\cdot)\in\mathcal{U}_{2},\quad i\in\mathcal{S};
    \end{aligned}
    \right.
  \end{equation}
  \item[(ii)] the FBSDEs \eqref{SLQ-FBSDE} admits a solution such that stationary condition \eqref{SLQ-stationary-condition}.
\end{description}

Let $X_{1}^{0}=X^{0}(\cdot;0,i,u_{1},0)$ and $X_{2}^{0}=X^{0}(\cdot;0,i,0,u_{2})$. Then
 \begin{align*}
\bar{J}_{k}^{0}(0,i;u_{k})&\triangleq \mathbb{E}\Big\{\int_{0}^{T}\left<\left(\begin{matrix}
Q(\alpha) & S_{k}(\alpha)^{\top}\\
S_{k}(\alpha) & R_{kk}(\alpha)\end{matrix}\right)\left(\begin{matrix}
X_{k}^{0}\\u_{k}\end{matrix}\right),
\left(\begin{matrix}
X_{k}^{0}\\u_{k}\end{matrix}\right)\right>dt
+\big<G(T,\alpha(T))X_{k}^{0}(T), X_{k}^{0}(T)\big>\Big\}\\
&=J^{0}(0,i;u_{k},0),\quad k=1,2,
  \end{align*}
 which implies that the condition \eqref{convex-concave} is equivalent to \eqref{convex-concave-condition}. This completes the proof.
\end{proof}

In the above, the FBSDEs \eqref{SLQ-FBSDE} together with the stationary condition \eqref{SLQ-stationary-condition} constitute the optimality system of Problem (M-ZLQ). Although Theorem \ref{thm-ZLQ-open-loop} has characterized the open-loop solvability of Problem (M-ZLQ) through the solvability of the optimality system and the concavity-convexity of the performance functional, further decoupling of the optimality system is required to obtain the closed-loop representation of the open-loop saddle point. The following result establishes a relation between the linear FBSDEs and the CDREs.
\begin{theorem}\label{thm-open-closed}
Let assumptions (A1)-(A2) hold. If %the following conditions hold:
\begin{description}
  \item[(i)] the convexity-concavity condition \eqref{convex-concave-condition} holds,
  \item[(ii)]  the CDREs \eqref{CDREs} admits a solution $\mathbf{P(\cdot)}\in\mathcal{D}\left(C(0, T; \mathbb{S}^{n}) \right)$ such that
  \begin{equation}\label{ZLQ-CDREs-open-closed-condition}
  \left\{
  \begin{array}{l}
  \mathcal{R}(\mathcal{L}(t;P,i)^{\top})\subseteq \mathcal{R}(\mathcal{N}(t;P,i)),\quad a.e. \text{ }t\in[0,T],\\[3mm]
      \mathcal{N}(\cdot;P,i)^{\dag}\mathcal{L}(\cdot;P,i)^{\top}\in L^{2}(0,T;\mathbb{R}^{m\times n}),\quad  \forall i\in\mathcal{S},
      \end{array}
      \right.
  \end{equation}
  \item[(iii)]  the adapted solution $(\eta(\cdot),\zeta(\cdot),\mathbf{z}(\cdot))$ to the BSDE \eqref{eta}
 satisfies the constraint \eqref{eta-constraint},
\end{description}
then for any $(x,i)\in\mathbb{R}^{n}\times\mathcal{S}$, the optimality system \eqref{SLQ-FBSDE}-\eqref{SLQ-stationary-condition} admits a solution. In addition, the Problem (M-ZLQ) is open-loop solvable with closed-loop representation:
\begin{equation}\label{open-closed-representation}
  u^{*}(\cdot;x,i)=[u_{1}^{*}(\cdot;x,i)^{\top},u_{2}^{*}(\cdot;x,i)^{\top}]^{\top}=\Theta^{*}(\cdot,\alpha(\cdot))X^{*}(\cdot;x,i)+\nu^{*}(\cdot),
\end{equation}
where
\begin{equation}\label{ZLQ-open-closed-representation}
 \left\{
 \begin{aligned}
 &\Theta^{*}(\cdot,i)\triangleq-\mathcal{N}(\cdot;P,i)^{\dag}\mathcal{L}(\cdot;P,i)^{\top},\quad i\in\mathcal{S},\\
 &\nu^{*}(\cdot)\triangleq -\mathcal{N}(\cdot;P,i)^{\dag} \widetilde{\rho},
 \end{aligned}
 \right.
\end{equation}
and $X^{*}(\cdot;x,i)\equiv X(\cdot;x,i,\mathbf{\Theta}^{*},\nu^{*})$ defined in \eqref{state-LQ-closed}.
%\begin{equation}\label{state-star-closed}
%  \left\{
%  \begin{aligned}
%  dX^{*}(t)&=\Big\{\big[A(t,\alpha(t))+B(t,\alpha(t))\Theta^{*}(t,\alpha(t))\big]X^{*}(t)+B(t,\alpha(t))\nu^{*}(t)+b(t)\Big\}dt\\
%  &\quad +\Big\{\big[C(t,\alpha(t))+D(t,\alpha(t))\Theta^{*}(t,\alpha(t))\big]X^{*}(t)+D(t,\alpha(t))\nu^{*}(t)+\sigma(t)\Big\}dW(t),\quad t\in[0,T],\\
%  X^{*}(0)&=x,\quad \alpha(0)=i.
%  \end{aligned}
%  \right.
%\end{equation}
\end{theorem}

\begin{proof}
 Let $\mathbf{P}(\cdot)$  be the solution to CDREs \eqref{CDREs} satisfying the condition \eqref{ZLQ-CDREs-open-closed-condition} and $(\eta(\cdot),\zeta(\cdot),\mathbf{z}(\cdot))$ be the solution to BSDE \eqref{eta} satisfying the condition \eqref{eta-constraint}. Now, for any fixed $(x,i)\in\mathbb{R}^{n}\times\mathcal{S}$, let
 \begin{equation}\label{Y-solution}
   \left\{
   \begin{aligned}
   &Y^{*}(t)\triangleq P(t,\alpha(t))X^{*}(t)+\eta(t),\\
   &Z^{*}(t)\triangleq P(t,\alpha(t))\big[C(t,\alpha(t))X^{*}(t)+D(t,\alpha(t))u^{*}(t)+\sigma(t)\big]+\zeta(t),\\
   &\Gamma_{j}^{*}(t)\triangleq \big[P(t,j)-P(t,\alpha(t-))\big]X^{*}(t)+z_{j}(t),\quad t\in [0,T],\quad j\in\mathcal{S}.
   \end{aligned}
   \right.
 \end{equation}
We only need to verify that $(X^{*}, Y^{*}, Z^{*}, \mathbf{\Gamma}^{*},u^{*})$ defined above solves the optimality system \eqref{SLQ-FBSDE}-\eqref{SLQ-stationary-condition}. The desired result follows from Theorem \ref{thm-ZLQ-open-loop} directly.

Noting that $\mathcal{L}(P,\alpha)^{\top}X^{*}+\widetilde{\rho}\in\mathcal{R}(\mathcal{N}(P,\alpha))$,
which implies that
\begin{equation}
    \begin{aligned}
    \mathcal{N}(P,\alpha)u^{*}&=-\mathcal{N}(P,\alpha)\mathcal{N}(P,\alpha)^{\dag}\mathcal{L}(P,\alpha)^{\top}X^{*}-\mathcal{N}(P,\alpha)\mathcal{N}(P,\alpha)^{\dag}\widetilde{\rho}\\
    &=\big[I-\mathcal{N}(P,\alpha)\mathcal{N}(P,\alpha)^{\dag}\big]\big[\mathcal{L}(P,\alpha)^{\top}X^{*}+\widetilde{\rho}\big]-\mathcal{L}(P,\alpha)^{\top}X^{*}-\widetilde{\rho}\\
    &=-\mathcal{L}(P,\alpha)^{\top}X^{*}-\widetilde{\rho}.
    \end{aligned}
\end{equation}
Consequently,
\begin{align*}
  &\quad B(\alpha)^{\top}Y^{*}+ D(\alpha)^{\top}Z^{*}+S(\alpha)X^{*}+R(\alpha)u^{*}+\rho\\
  &=B(\alpha)^{\top}(P(\alpha)X^{*}+\eta)+ D(\alpha)^{\top}\big[P(\alpha)\big(C(\alpha)X^{*}+D(\alpha)u^{*}+\sigma\big)+\zeta\big]+S(\alpha)X^{*}+R(\alpha)u^{*}+\rho\\
  &=\mathcal{N}(P,\alpha)u^{*}+\mathcal{L}(P,\alpha)^{\top}X^{*}+\widetilde{\rho}\\
  &=0.
\end{align*}

On the other hand, it follows from \eqref{open-closed-representation} that $X^{*}(\cdot)$ also solves the following SDE:
\begin{equation}
  \left\{
      \begin{aligned}
      dX^{*}(t)&=\left[A(t,\alpha(t))X^{*}(t)+B(t,\alpha(t))u^{*}(t)+b(t)\right]dt\\
      &\quad+\left[C(t,\alpha(t))X^{*}(t)+D(t,\alpha(t))u^{*}(t)+\sigma(t)\right]dW(t),\quad t\in [0,T],\\
      X^{*}(0)&=x,\quad\alpha_{0}=i.
      \end{aligned}
      \right.
\end{equation}

Let
\begin{align*}
\Lambda &\triangleq -\big\{\big[A(\alpha)^{\top}-\mathcal{L}(P,\alpha)\mathcal{N}(P,\alpha)^{\dag}B(\alpha)^{\top}\big]\eta+\big[C(\alpha)^{\top}-\mathcal{L}(P,\alpha)\mathcal{N}(P,\alpha)^{\dag}D(\alpha)^{\top}\big]\zeta\\
        &\quad+\big[C(\alpha)^{\top}-\mathcal{L}(P,\alpha)\mathcal{N}(P,\alpha)^{\dag}D(\alpha)^{\top}\big]P(\alpha)\sigma-\mathcal{L}(P,\alpha)\mathcal{N}(P,\alpha)^{\dag}\rho+P(\alpha)b+q\big\}.
\end{align*}
Then applying It\^o's rule to $P(\alpha)X^{*}+\eta$, one has
\begin{align*}
  dY^{*}&=\big\{P(\alpha)\big[A(\alpha)X^{*}+B(\alpha)u^{*}+b\big]+\dot{P}(\alpha)X^{*}+\sum_{j=1}^{L}\pi_{\alpha_{t-},j}P(j)X^{*}+\Lambda\big\}dt+Z^{*}(t)dW(t)+\Gamma^{*}(t)\cdot d\widetilde{N}(t) \\
  &=\big\{\big[\mathcal{L}(P,\alpha) \mathcal{N}(P,i)^{\dag} \mathcal{L}(P,\alpha)^{\top}-A(\alpha)^{\top}P(\alpha)-C(\alpha)^{\top}P(\alpha)C(\alpha)-Q(\alpha)\big]X^{*}+P(\alpha)B(\alpha)u^{*}\\
  &\quad +P(\alpha)b+\Lambda\big\}dt+Z^{*}(t)dW(t)+\Gamma^{*}(t)\cdot d\widetilde{N}(t) \\
  &=\big\{\mathcal{L}(P,i) \mathcal{N}(P,i)^{\dag} \mathcal{L}(P,i)^{\top}X^{*}-A(\alpha)^{\top}\big(Y^{*}-\eta\big)-C(\alpha)^{\top}\big[Z^{*}-P(\alpha)D(\alpha)u^{*}-P(\alpha)\sigma-\zeta\big]\\
 &\quad -Q(\alpha)X^{*}+P(\alpha)B(\alpha)u^{*} +P(\alpha)b+\Lambda\big\}dt+Z^{*}(t)dW(t)+\Gamma^{*}(t)\cdot d\widetilde{N}(t) \\
 &=-\big\{A(\alpha)^{\top}Y^{*}+C(\alpha)^{\top}Z^{*}+Q(\alpha)X^{*}+S(\alpha)^{\top}u^{*}
 -\mathcal{L}(P,i)\big[u^{*}+\mathcal{N}(P,i)^{\dag}\mathcal{L}(P,i)^{\top}X^{*}\big]\\
 &\quad -\big[A(\alpha)^{\top}\eta+C(\alpha)^{\top}\zeta+C(\alpha)^{\top}P(\alpha)\sigma+P(\alpha)b+\Lambda\big]\big\}dt+Z^{*}(t)dW(t)+\Gamma^{*}(t)\cdot d\widetilde{N}(t) \\
 &=-\big\{A(\alpha)^{\top}Y^{*}+C(\alpha)^{\top}Z^{*}+Q(\alpha)X^{*}+S(\alpha)^{\top}u^{*}+q\\
 &\quad-\mathcal{L}(P,\alpha)\big[u^{*}+\mathcal{N}(P,\alpha)^{\dag}\mathcal{L}(P,\alpha)^{\top}X^{*}
 +\mathcal{N}(P,\alpha)^{\dag}\widetilde{\rho}\big]\big\}dt+Z^{*}(t)dW(t)+\Gamma^{*}(t)\cdot d\widetilde{N}(t) \\
 &=-\big\{A(\alpha)^{\top}Y^{*}+C(\alpha)^{\top}Z^{*}+Q(\alpha)X^{*}+S(\alpha)^{\top}u^{*}+q\big]\big\}dt+Z^{*}(t)dW(t)+\Gamma^{*}(t)\cdot d\widetilde{N}(t).
\end{align*}
Combining with the fact that
$$Y^{*}(T)=P(T,\alpha(T))X^{*}(T)+\eta^{*}(T)=G(\alpha(T))X^{*}(T)+g,$$
we complete the proof.
\end{proof}

\section{Closed-loop solvability}\label{section-4}
We begin this section by letting:
$$\mathcal{V}_{cl}\triangleq \mathcal{D}\left(L^{2}(0,T;\mathbb{R}^{m_{1}\times n})\right)\times \mathcal{U}_{1}\times \mathcal{D}\left(L^{2}(0,T;\mathbb{R}^{m_{2}\times n})\right)\times \mathcal{U}_{2},$$
which represents the set of all closed-loop strategies of Problem (M-ZLQ). To simplify our further analysis, for given $\mathbf{P}(\cdot)\in\mathcal{D}\left(C(0,T;\mathbb{S}^{n})\right)$, we denote
\begin{equation}\label{ZLQ-closed-notations}
\begin{aligned}
    &\mathcal{L}_{k}(\cdot;P,i)\triangleq P(\cdot,i)B_{k}(\cdot,i)+C(\cdot,i)^{\top}P(\cdot,i)D_{k}(\cdot,i)+S_{k}(\cdot,i)^{\top},\\
    &\mathcal{N}_{kl}(\cdot;P,i)\triangleq
    D_{k}(\cdot,i)^{\top}P(\cdot,i)D_{l}(\cdot,i)+R_{kl}(\cdot,i),\quad k,l\in\{1,2\}.
    \end{aligned}
\end{equation}
Clearly, it follows from notations \eqref{natation-1} and \eqref{notation-LMN} that
\begin{equation}\label{ZLQ-notation-LN}
    \mathcal{L}(\cdot;P,i)=\big(\mathcal{L}_{1}(\cdot;P,i),\mathcal{L}_{2}(\cdot;P,i)\big),\quad
    \mathcal{N}(\cdot;P,i)=\left(\begin{matrix}
    \mathcal{N}_{11}(\cdot;P,i)&\mathcal{N}_{12}(\cdot;P,i)\\
    \mathcal{N}_{21}(\cdot;P,i)&\mathcal{N}_{22}(\cdot;P,i)
    \end{matrix}\right).
\end{equation}

For given
$\mathbf{\Theta}(\cdot)\equiv(\mathbf{\Theta}_{1}(\cdot)^{\top},\mathbf{\Theta}_{2}(\cdot)^{\top})\in\mathcal{D}(L^{2}(0,T;\mathbb{R}^{m\times n}))$ and $\nu(\cdot)\equiv (\nu_{1}(\cdot)^{\top},\nu_{2}(\cdot)^{\top})^{\top}\in L_{\mathbb{F}}^{2}(\mathbb{R}^{m})$, we denote  the solution to the SDE \eqref{state-LQ-closed}  as $X(\cdot;x,i,\mathbf{\Theta},\nu)$ and  as $X^{0}(\cdot;x,i,\mathbf{\Theta},\nu)$ if $b=\sigma=0$. Specifically, for given $u_{k}(\cdot)\in \mathcal{U}_{k}$, $\mathbf{\Theta}_{l}(\cdot)\in\mathcal{D}(L^{2}(0,T;\mathbb{R}^{m_{l}\times n}))$ and $\nu_{l}(\cdot)\in \mathcal{U}_{l}$, let $X(\cdot;x,i,u_{k},\mathbf{\Theta}_{l},\nu_{l})$ be the solution to the following  SDE:
$$
    \left\{
    \begin{aligned}
   dX(t)&=\left[\big(A(t,\alpha(t))+B_{l}(t,\alpha(t))\Theta_{l}(t,\alpha(t))\big)X(t)+B_{l}(t,\alpha(t))\nu_{l}(t)+B_{k}(t,\alpha(t))u_{k}(t)+b(t)\right]dt\\
   &\quad+\left[\big(C(t,\alpha(t))+D_{l}(t,\alpha(t))\Theta_{l}(t,\alpha(t))\big)X(t)+D_{l}(t,\alpha(t))\nu_{l}(t)+D_{k}(t,\alpha(t))u_{k}(t)+\sigma(t)\right]dW(t),\\
   X(0)&=x,\quad \alpha(0)=i, \quad (k,l)\in\{(1,2), (2,1)\}.
    \end{aligned}
    \right.
$$
If $b=\sigma=0$, we denote the  solution to above SDE as $X^{0}(\cdot;x,i,u_{k},\mathbf{\Theta}_{l},\nu_{l})$.
For $(k,l)\in\{(1,2), (2,1)\}$, we define
\begin{equation}\label{cost-functional-ZLQ}
\left\{
\begin{aligned}
    &J(x,i;u_{k},\mathbf{\Theta}_{l},\nu_{l})\triangleq J(x,i;u_{k},\Theta_{l}(\alpha)X+\nu_{l})\quad \text{ with }\quad X\equiv X(\cdot;x,i,u_{k},\mathbf{\Theta}_{l},\nu_{l}),\\
    &J(x,i;\mathbf{\Theta},\nu)\triangleq J(x,i;\Theta_{1}(\alpha)X+\nu_{1},\Theta_{2}(\alpha)X+\nu_{2})\quad  \text{ with }\quad X\equiv X(\cdot;x,i,\mathbf{\Theta},\nu),\\
    &J^{0}(x,i;u_{k},\mathbf{\Theta}_{l},\nu_{l})\triangleq J^{0}(x,i;u_{k},\Theta_{l}(\alpha)X^{0}+\nu_{l})\quad \text{ with }\quad X^{0}\equiv X^{0}(\cdot;x,i,u_{k},\mathbf{\Theta}_{l},\nu_{l}),\\
    &J^{0}(x,i;\mathbf{\Theta},\nu)\triangleq J^{0}(x,i;\Theta_{1}(\alpha)X^{0}+\nu_{1},\Theta_{2}(\alpha)X^{0}+\nu_{2})\quad  \text{ with }\quad X^{0}\equiv X^{0}(\cdot;x,i,\mathbf{\Theta},\nu).
\end{aligned}\right.
\end{equation}
Then, the closed-loop Nash equilibrium strategy of the Problem (M-ZLQ) can be defined based on the above notations.
\begin{definition}\label{def-ZLQ-closed-loop-equilibrium}
A 4-tuple $(\mathbf{\widehat{\Theta}_{1}}(\cdot),\widehat{\nu}_{1}(\cdot);\mathbf{\widehat{\Theta}_{2}}(\cdot),
\widehat{\nu}_{2}(\cdot))\in\mathcal{V}_{cl}$
is called a closed-loop Nash equilibrium strategy of Problem (M-ZLQ) if for any $(x,i)\in\mathbb{R}^{n}\times\mathcal{S}$, the following holds:
  \begin{equation}\label{closed-loop-Nash-equilibrium-point}
J(x,i;\mathbf{\widehat{\Theta}}_{1},\widehat{\nu}_{1},u_{2})\leq J(x,i;\mathbf{\widehat{\Theta}},\widehat{\nu})\leq J(x,i;u_{1},\mathbf{\widehat{\Theta}}_{2},\widehat{\nu}_{2}),\quad
\forall (u_{1},u_{2})\in\mathcal{U}_{1}\times \mathcal{U}_{2},
  \end{equation}
where
$$
\widehat{\Theta}(\cdot,i)\triangleq (\widehat{\Theta}_{1}(\cdot,i)^{\top},\widehat{\Theta}_{2}(\cdot,i)^{\top})^{\top},\quad
\widehat{\nu}(\cdot)\triangleq (\widehat{\nu}_{1}(\cdot)^{\top},\widehat{\nu}_{2}(\cdot)^{\top})^{\top},\quad  i\in\mathcal{S}.
$$
 The problem (M-ZLQ) is said to be (uniquely) closed-loop solvable if it admits a (unique) closed-loop Nash equilibrium strategy.
In this case, we denote
\begin{equation}\label{ZLQ-outcome-closed}
\left\{
\begin{aligned}
&\widehat{u}_{1}(\cdot;x,i,\mathbf{\widehat{\Theta}},\widehat{\nu})\triangleq \widehat{\Theta}_{1}(\cdot,\alpha(\cdot))X(\cdot;x,i,\mathbf{\widehat{\Theta}},\widehat{\nu})+\widehat{\nu}_{1}(\cdot),\\
&\widehat{u}_{2}(\cdot;x,i,\mathbf{\widehat{\Theta}},\widehat{\nu})\triangleq \widehat{\Theta}_{2}(\cdot,\alpha(\cdot))X(\cdot;x,i,\mathbf{\widehat{\Theta}},\widehat{\nu})+\widehat{\nu}_{2}(\cdot),\\
\end{aligned}
\right.
\end{equation}
as the outcome of the closed-loop Nash equilibrium strategy $(\mathbf{\widehat{\Theta}},\widehat{\nu})$ and denote
$$V(x,i)=J(x,i;\widehat{\mathbf{\Theta}},\widehat{\nu}),$$
as the closed-loop value function of Problem (M-ZLQ).
\end{definition}

\begin{remark}\rm
  In \eqref{closed-loop-Nash-equilibrium-point}, the state process $X(\cdot; x,i, \mathbf{\widehat{\Theta}}_{1},\widehat{\nu}_{1},u_{2})$ appearing in $J(x,i;\mathbf{\widehat{\Theta}}_{1},\widehat{\nu}_{1},u_{2})$ is different from $X(\cdot; x,i,u_{1}, \mathbf{\widehat{\Theta}}_{2},\widehat{\nu}_{2})$  in $J(x,i;u_{1},\mathbf{\widehat{\Theta}}_{2},\widehat{\nu}_{2})$, and both are different from $X(\cdot; x,i, \mathbf{\widehat{\Theta}},\widehat{\nu})$. Therefore, comparing with \eqref{ZLQ-open-loop-Nash-equilibrium-point}, we see that \eqref{closed-loop-Nash-equilibrium-point}  does not imply that the outcome $(\widehat{u}_{1}(\cdot;x,i,\mathbf{\widehat{\Theta}},\widehat{\nu}),\widehat{u}_{2}(\cdot;x,i,\mathbf{\widehat{\Theta}},\widehat{\nu}))$ of the closed-loop Nash equilibrium strategy is an open-loop saddle point of Problem (M-ZLQ).
\end{remark}

For a given 4-tuple $(\mathbf{\widehat{\Theta}_{1}}(\cdot),\widehat{\nu}_{1}(\cdot);\mathbf{\widehat{\Theta}_{2}}(\cdot),
\widehat{\nu}_{2}(\cdot))\in\mathcal{V}_{cl}$ and $(u_{1}(\cdot),u_{2}(\cdot))\in\mathcal{U}_{1}\times\mathcal{U}_{2}$, let $\widehat{X}_{k}(\cdot;x,i,u_{k})$ be the solution to the following SDE:
\begin{equation}\label{ZLQ-closed-state-k}
    \left\{
    \begin{aligned}
    &d\widehat{X}_{k}(t)=\big[\widehat{A}_{k}(\alpha)\widehat{X}_{k}+B_{k}(\alpha)u_{k}+\widehat{b}_{k}\big]dt
    +\big[\widehat{C}_{k}(\alpha)\widehat{X}_{k}+D_{k}(\alpha)u_{k}+\widehat{\sigma}_{k}\big]dW(t),\quad t\in[0,T],\\
    &\widehat{X}_{k}(0)=x,\quad i\in\mathcal{S},
    \end{aligned}
    \right.
\end{equation}
and
\begin{equation}\label{cost-ZLQ-k}
\begin{aligned}
    \widehat{J}_{k}\left(x,i;u_{k}\right)
    & = \mathbb{E}\Big\{\int_{0}^{\infty}\left[
    \left<
    \left(
    \begin{matrix}
    \widehat{Q}_{k}(\alpha) & \widehat{S}_{k}(\alpha)^{\top} \\
    \widehat{S}_{k}(\alpha) & R_{kk}(\alpha)
    \end{matrix}
    \right)
    \left(
    \begin{matrix}
    \widehat{X}_{k} \\
    u_{k}
    \end{matrix}
    \right),
    \left(
    \begin{matrix}
    \widehat{X}_{k} \\
    u_{k}
    \end{matrix}
    \right)
    \right>
    +2\left<
    \left(
    \begin{matrix}
    \widehat{q}_{k} \\
    \widehat{\rho}_{k}
    \end{matrix}
    \right),
    \left(
    \begin{matrix}
    \widehat{X}_{k}\\
    u_{k}
    \end{matrix}
    \right)
    \right>\right]dt\\
    &\quad +\big<G(\alpha(T))X_{k}(T)+2g,X_{k}(T)\big>\Big\},\quad k=1,2,
  \end{aligned}
\end{equation}
where for $i\in\mathcal{S}$,
\begin{equation}\label{ZLQ-notation-state-cost}
    \left\{
\begin{aligned}
&\widehat{A}_{1}(\cdot,i)=A(\cdot,i)+B_{2}(\cdot,i)\widehat{\Theta}_{2}(\cdot,i),\quad
\widehat{b}_{1}(\cdot)=B_{2}(\cdot,\alpha(\cdot))\widehat{\nu}_{2}(\cdot)+b(\cdot),\\
&\widehat{Q}_{1}(\cdot,i)= Q(\cdot,i)+S_{2}(\cdot,i)^{\top}\widehat{\Theta}_{2}(\cdot,i)
+\widehat{\Theta}_{2}(\cdot,i)^{\top}S_{2}(\cdot,i)
+\widehat{\Theta}_{2}(\cdot,i)^{\top}R_{22}(\cdot,i)\widehat{\Theta}_{2}(\cdot,i),\\
&\widehat{S}_{1}(\cdot,i)=S_{1}(\cdot,i)+R_{12}(\cdot,i)\widehat{\Theta}_{2}(\cdot,i),\quad
\widehat{\rho}_{1}(\cdot)=\rho_{1}(\cdot)+R_{12}(\cdot,\alpha(\cdot))\widehat{\nu}_{2}(\cdot),\\
&\widehat{q}_{1}(\cdot)=q(\cdot)+\widehat{\Theta}_{2}(\cdot,\alpha(\cdot))^{\top}\rho_{2}(\cdot)
+\big[S_{2}(\cdot,\alpha(\cdot))^{\top}+\widehat{\Theta}_{2}(\cdot,\alpha(\cdot))^{\top}R_{22}(\cdot,\alpha(\cdot))\big]\widehat{\nu}_{2}(\cdot),\\
&\widehat{A}_{2}(\cdot,i)=A(\cdot,i)+B_{1}(\cdot,i)\widehat{\Theta}_{1}(\cdot,i),\quad
\widehat{b}_{2}(\cdot)=B_{1}(\cdot,\alpha(\cdot))\widehat{\nu}_{1}(\cdot)+b(\cdot),\\
&\widehat{Q}_{2}(\cdot,i)= Q(\cdot,i)+S_{1}(\cdot,i)^{\top}\widehat{\Theta}_{1}(\cdot,i)
+\widehat{\Theta}_{1}(\cdot,i)^{\top}S_{1}(\cdot,i)
+\widehat{\Theta}_{1}(\cdot,i)^{\top}R_{11}(\cdot,i)\widehat{\Theta}_{1}(\cdot,i),\\
&\widehat{S}_{2}(\cdot,i)=S_{2}(\cdot,i)+R_{21}(\cdot,i)\widehat{\Theta}_{1}(\cdot,i),\quad
\widehat{\rho}_{2}(\cdot)=\rho_{2}(\cdot)+R_{21}(\cdot,\alpha(\cdot))\widehat{\nu}_{1}(\cdot),\\
&\widehat{q}_{2}(\cdot)=q(\cdot)+\widehat{\Theta}_{1}(\cdot,\alpha(\cdot))^{\top}\rho_{1}(\cdot)
+\big[S_{1}(\cdot,\alpha(\cdot))^{\top}+\widehat{\Theta}_{1}(\cdot,\alpha(\cdot))^{\top}R_{11}(\cdot,\alpha(\cdot))\big]\widehat{\nu}_{1}(\cdot).
\end{aligned}
\right.
\end{equation}
%For $k=1,2$, if $\widehat{b}_{k}=\widehat{\sigma}_{k}=\widehat{q}_{k}=0$, $\widehat{\rho}_{k}=0$, $g=0$, then we donate the corresponding state equation and performance functional as $\widehat{X}_{k}^{0}(\cdot;x,i,u_{k})$ and $\widehat{J}_{k}^{0}(x,i;u_{k})$, respectively. In addition, one can defined $\widehat{X}_{k}(\cdot;x,i,\mathbf{\Theta}_{k},\nu_{k})$ (respectively, $\widehat{X}_{k}^{0}(\cdot;x,i,\mathbf{\Theta}_{k},\nu_{k})$) and $\widehat{J}_{k}(x,i;\mathbf{\Theta}_{k},\nu_{k})$ (respectively, $\widehat{J}_{k}^{0}(x,i;\mathbf{\Theta}_{k},\nu_{k})$) similar to \eqref{state-LQ-closed}-\eqref{SLQ-cost-closed}.
Clearly, by some straightforward calculations, one has:
\begin{equation}\label{ZLQ-cost-closed}
    \begin{aligned}
        & J\big(x,i;u_{1},\mathbf{\widehat{\Theta}}_{2},\widehat{\nu}_{2}\big)
        =\widehat{J}_{1}\left(x,i;u_{1}\right)
        +\mathbb{E}\int_{0}^{T}\big<R_{22}(\alpha)\widehat{\nu}_{2}+2\rho_{2},\widehat{\nu}_{2}\big>dt,\\
        &J\big(x,i;\mathbf{\widehat{\Theta}}_{1},\widehat{\nu}_{1},u_{2}\big)
        =\widehat{J}_{2}\left(x,i;u_{2}\right)
        +\mathbb{E}\int_{0}^{T}\big<R_{11}(\alpha)\widehat{\nu}_{1}+2\rho_{1},\widehat{\nu}_{1}\big>dt,
    \end{aligned}
\end{equation}

%In the following, we define
%$$\widehat{J}_{k}\left(x,i;\mathbf{\Theta}_{k},\nu_{k}\right)
%=\widehat{J}_{k}\left(x,i;u_{k}\right)\quad \text{with} \quad u_{k}(\cdot)
%=\Theta_{k}(\cdot,\alpha(\cdot))\widehat{X}_{k}(\cdot;x,i,\mathbf{\Theta}_{k},\nu_{k})+\nu_{k}(\cdot),$$
%where $\widehat{X}_{k}(\cdot;x,i,\mathbf{\Theta}_{k},\nu_{k})$ solves the following SDE:
%\begin{equation}\label{ZLQ-closed-state-k-2}
%    \left\{
%    \begin{aligned}
%    d\widehat{X}_{k}(t)&=\Big\{\big[\widehat{A}_{k}(t,\alpha(t))+B_{k}(t,\alpha(t))\Theta_{k}(t,\alpha(t))\big]\widehat{X}_{k}(t)
%    +B_{k}(t,\alpha(t))\nu_{k}(t)+\widehat{b}_{k}(t)\Big\}dt\\
%    &\quad+\Big\{\big[\widehat{C}_{k}(t,\alpha(t))+D_{k}(t,\alpha(t))\Theta_{k}(t,\alpha(t))\big]\widehat{X}_{k}(t)
%    +D_{k}(t,\alpha(t))\nu_{k}(t)+\widehat{\sigma}_{k}(t)\big]dW(t),\\
%    \widehat{X}_{k}(0)&=x,\quad i\in\mathcal{S}.
%    \end{aligned}
%    \right.
%\end{equation}
%If $\widehat{b}_{k}=\widehat{\sigma}_{k}=\widehat{q}_{k}=0$, then one can define the $\widehat{X}_{k}^{0}(\cdot;x,i,\mathbf{\Theta}_{k},\nu_{k})$ and $\widehat{J}_{k}^{0}\left(x,i;\mathbf{\Theta}_{k},\nu_{k}\right) $ similarly.

Now,  we introduce the following two coupled SLQ control problems for studying the closed-loop solvability of Problem (M-ZLQ).\\
\noindent\textbf{Problem ($\widehat{\text{M-SLQ}}$)$_{1}$:} For any given $(x,i)\in \mathbb{R}^{n}\times \mathcal{S}$, find a $u_{1}^{*}\in \mathcal{U}_{1}$ such that
\begin{equation}
\widehat{J}_{1}\left(x,i;u_{1}^{*}\right)=\inf_{u_{1}\in \mathcal{U}_{1}}\widehat{J}_{1}\left(x,i;u_{1}\right)\triangleq\widehat{V}_{1}(x,i).
\end{equation}
\noindent\textbf{Problem ($\widehat{\text{M-SLQ}}$)$_{2}$:} For any given $(x,i)\in \mathbb{R}^{n}\times \mathcal{S}$, find a $u_{2}^{*}\in \mathcal{U}_{2}$ such that
\begin{equation}
\widehat{J}_{2}\left(x,i;u_{2}^{*}\right)=\sup_{u_{2}\in \mathcal{U}_{2}}\widehat{J}_{1}\left(x,i;u_{2}\right)\triangleq\widehat{V}_{2}(x,i).
\end{equation}
In the above, the function $\widehat{V}_{k}(\cdot,\cdot)$ is called the value function of Problem ($\widehat{\text{M-SLQ}}$)$_{k}$, $k=1,2$. If $\widehat{b}_{k}=\widehat{\sigma}_{k}=\widehat{q}_{k}=0$, $\widehat{\rho}_{k}=0$, $g=0$, then we donate the corresponding value function as $\widehat{V}_{k}^{0}(\cdot,\cdot)$.
%If $\widehat{b}_{k}=\widehat{\sigma}_{k}=\widehat{q}_{k}=0$, $\widehat{\rho}_{k}=0$, then we
% denote the Problem ($\widehat{\text{M-SLQ}}$)$_{k}$ as Problem ($\widehat{\text{M-SLQ}}$)$_{k}^{0}$ and denote the performance functional of Problem ($\widehat{\text{M-SLQ}}$)$_{k}^{0}$ as $\widehat{J}_{k}^{0}\left(x,i;u_{k}\right)$.

By Definition \ref{def-ZLQ-closed-loop-equilibrium}, a 4-tuple $(\mathbf{\widehat{\Theta}_{1}},\widehat{\nu}_{1};\mathbf{\widehat{\Theta}_{2}},\widehat{\nu}_{2})$ is a closed-loop Nash equilibrium strategy of Problem (M-ZLQ) if and only if $(\mathbf{\widehat{\Theta}_{k}},\widehat{\nu}_{k})$ is a closed-loop optimal control of Problem ($\widehat{\text{M-SLQ}}$)$_{k}$ for $k=1,2$.
The following result provides a deeper characterization of the closed-loop Nash equilibrium strategy of the Problem (M-ZLQ).

\begin{theorem}\label{thm-ZLQ-closed-loop}
A 4-tuple $(\mathbf{\widehat{\Theta}_{1}}(\cdot),\widehat{\nu}_{1}(\cdot);\mathbf{\widehat{\Theta}_{2}}(\cdot),
\widehat{\nu}_{2}(\cdot))
\in\mathcal{V}_{cl}$ is a closed-loop Nash equilibrium strategy of Problem (M-ZLQ)  if and only if:
 \begin{description}
  \item[(i)] The CDREs \eqref{CDREs} admits a unique solution $\mathbf{P}(\cdot)\in\mathcal{D}\left(C(0,T;\mathbb{S}^{n})\right)$ such that
      \begin{equation}\label{ZLQ-closed-loop-constraint}
        \left\{
        \begin{aligned}
        &\mathcal{N}(t;P,i)\widehat{\Theta}(t,i)+\mathcal{L}(t;P,i)^{\top}=0, \quad a.e.\text{ }t\in[0,T],\\
        &\mathcal{N}_{11}(\cdot;P,i)\geq 0,\quad \mathcal{N}_{22}(\cdot;P,i)\leq 0,
        \end{aligned}
        \right.
      \end{equation}
      where
      $\widehat{\Theta}(\cdot,i)
    \equiv(\widehat{\Theta}_{1}(\cdot,i)^{\top},\widehat{\Theta}_{2}(\cdot,i)^{\top})^{\top}$ for any $ i\in\mathcal{S}$;
  \item[(ii)] The adapted solution $(\eta(\cdot),\zeta(\cdot),\mathbf{z}(\cdot))$ to BSDE \eqref{eta} satisfies
      \begin{equation}\label{ZLQ-eta-closed-constraint}
       \mathcal{N}(t;P,\alpha)\widehat{\nu}(t)+\widetilde{\rho}(t)=0,\quad a.e.,\text{ }t\in[0,T],
      \end{equation}
      where $\widetilde{\rho}(\cdot)$ is defined in \eqref{SLQ-rho} and
      $\widehat{\nu}(\cdot) \equiv (\widehat{\nu}_{1}(\cdot)^{\top},\widehat{\nu}_{2}(\cdot)^{\top})^{\top}$.
\end{description}

In this case, the closed-loop value function is given by
\begin{equation}\label{ZLQ-value}
      V(x,i)=\mathbb{E}\Big\{\big<P(0,i)x,x\big>+2\big<\eta(0),x\big>+\int_{0}^{T}\big[
      \big<P(\alpha)\sigma,\sigma\big>+2\big<\eta,b\big>+2\big<\zeta,\sigma\big>
      -\big<\mathcal{N}(P,\alpha)^{\dag}\widetilde{\rho},\widetilde{\rho}\big>\big]dt\Big\}.
  \end{equation}
\end{theorem}

\begin{proof}
For given $\mathbf{P}(\cdot)\in\mathcal{D}\left(C(0,T;\mathbb{S}^{n})\right)$, we define
\begin{equation*}\label{notation-hat-LMN}
\begin{array}{l}
\widehat{\mathcal{M}}_{k}(t;P,i)\triangleq P(t,i)\widehat{A}_{k}(t,i)+\widehat{A}_{k}(t,i)^{\top}P(t,i)+\widehat{C}_{k}(t,i)^{\top}P(t,i)\widehat{C}_{k}(t,i)
+\widehat{Q}_{k}(t,i)+\sum_{j=1}^{L}\pi_{ij}(t)P(t,j),\\[2mm]
\widehat{\mathcal{L}}_{k}(t;P,i)\triangleq P(t,i)B_{k}(t,i)+\widehat{C}_{k}(t,i)^{\top}P(t,i)D_{k}(t,i)+\widehat{S}_{k}(t,i)^{\top},\\[2mm]
\widehat{\mathcal{N}}_{k}(t;P,i)\triangleq D_{k}(t,i)^{\top}P(t,i)D_{k}(t,i)+R_{kk}(t,i),\quad t\in[0,T],\quad i\in\mathcal{S}.
\end{array}
\end{equation*}
Then by Lemma \ref{thm-SLQ-closed} and Remark \ref{rmk-SLQ-closed-2},  a 4-tuple $(\mathbf{\widehat{\Theta}_{1}},\widehat{\nu}_{1};\mathbf{\widehat{\Theta}_{2}},\widehat{\nu}_{2})$ is a closed-loop Nash equilibrium strategy of Problem (M-ZLQ) if and only if (for $k=1,2$):
\begin{description}
\item[(i)] The CDREs
\begin{equation}\label{ZLQ-CDREs-k}
    \left\{
    \begin{aligned}
     &\dot{P}_{k}(t,i)+\widehat{\mathcal{M}}_{k}(t;P_{k},i)-\widehat{\mathcal{L}}_{k}(t;P_{k},i)\widehat{\mathcal{N}}_{k}(t;P_{k},i)^{\dag}\widehat{\mathcal{L}}_{k}(t;P_{k},i)^{\top}=0,\\
    %&\mathcal{L}(t;P,i)\big[I-\mathcal{N}(t;P,i)\mathcal{N}(t;P,i)^{\dag}\big]=0,\quad t\in[0,T],\\
    &P_{k}(T,i)=G(i),\quad i\in\mathcal{S},
    \end{aligned}
    \right.
\end{equation}
admits a unique solution $\mathbf{P}_{k}(\cdot)\in\mathcal{D}\left(C(0,T;\mathbb{S}^{n})\right)$ such that
 \begin{equation}\label{ZLQ-Riccati-constraint-1}
 \left\{
 \begin{aligned}
       &\widehat{\mathcal{N}}_{k}(t,P_{k},i)\widehat{\Theta}_{k}(t,i)+\widehat{\mathcal{L}}_{k}(t,P_{k},i)^{\top}=0,\\
       &(-1)^{k+1}\widehat{\mathcal{N}}_{k}(t,P_{k},i)\geq 0,\quad a.e.\text{ }t\in[0,T];
\end{aligned}
\right.
   \end{equation}
\item[(ii)] The solution $(\eta_{k}(\cdot),\zeta_{k}(\cdot),\mathbf{z}_{k}(\cdot))$ to the following BSDE
\begin{equation}\label{ZLQ-eta-k}
       \left\{
      \begin{aligned}
        d\eta_{k}(t)&=-\big\{\big[\widehat{A}_{k}(t,\alpha(t))^{\top}-\widehat{\mathcal{L}}_{k}(t,P_{k},\alpha(t))\widehat{\mathcal{N}}_{k}(t,P_{k},\alpha(t))^{\dag}B_{k}(t,\alpha(t))^{\top}\big]\eta_{k}(t)\\
        &\quad+\big[\widehat{C}_{k}(t,\alpha(t))^{\top}-\widehat{\mathcal{L}}_{k}(t,P_{k},\alpha(t))\widehat{\mathcal{N}}_{k}(t,P_{k},\alpha(t))^{\dag}D_{k}(t,\alpha(t))^{\top}\big]\zeta_{k}(t)\\
        &\quad+\big[\widehat{C}_{k}(t,\alpha(t))^{\top}-\widehat{\mathcal{L}}_{k}(t,P_{k},\alpha(t))\widehat{\mathcal{N}}_{k}(t,P_{k},\alpha(t))^{\dag}D_{k}(t,\alpha(t))^{\top}\big]P_{k}(t,\alpha(t))\widehat{\sigma}_{k}(t)\\
        &\quad-\widehat{\mathcal{L}}_{k}(t,P_{k},\alpha(t))\widehat{\mathcal{N}}_{k}(t,P_{k},\alpha(t))^{\dag}\widehat{\rho}_{k}+P_{k}(t,\alpha(t))\widehat{b}_{k}(t)+\widehat{q}_{k}(t)\big\}dt\\
        &\quad+\zeta_{k}(t) dW(t)+\mathbf{z}_{k}(t)\cdot d\mathbf{\widetilde{N}}(t),\quad t\in [0,T],\\
        \eta_{k}(T)&=g,
      \end{aligned}
      \right.
   \end{equation}
   satisfies
   \begin{equation}\label{ZLQ-eta-k-constraint}
       \widehat{\mathcal{N}}_{k}(t,P_{k},\alpha)\widehat{\nu}_{k}(t)+\widehat{\varrho}_{k}(t)=0,\quad a.e.\text{ }t\in[0,T],
   \end{equation}
   where
   \begin{equation}\label{widehat-varrho-k}
   \begin{array}{r}
       \widehat{\varrho}_{k}(t)\triangleq B_{k}(t,\alpha(t))^{\top}\eta_{k}(t)+D_{k}(t,\alpha(t))^{\top}\zeta_{k}(t)+D_{k}(t,\alpha(t))^{\top}P_{k}(t,\alpha(t))\widehat{\sigma}_{k}(t)+\widehat{\rho}_{k}(t),\\[2mm]
       \quad a.e.\text{ }t\in[0,T].
       \end{array}
   \end{equation}
\end{description}

Noting that for any $i\in\mathcal{S}$,
$$
\big<P_{1}(0,i)x,x\big>=\widehat{V}_{1}^{0}(x,i)=J^{0}(x,i;\widehat{\mathbf{\Theta}},0)=\widehat{V}_{2}^{0}(x,i)=\big<P_{2}(0,i)x,x\big>,\quad \forall x\in\mathbb{R}^{n},
$$
which implies $P_{1}(0,i)=P_{2}(0,i)$. In fact, one can further obtain that $P_{1}(t,i)=P_{2}(t,i)$ for any $t\in[0,T]$ by considering the Problem ($\widehat{\text{M-ZLQ}}$)$_{k}^{0}$ over the horizon $[t,T]$.  Without loss of generality, for $i\in\mathcal{S}$, let $P(\cdot,i)\triangleq P_{1}(\cdot,i)=P_{2}(\cdot,i)$. Then, by some straightforward calculations,  one has
\begin{equation}\label{ZLQ-MLN-relation}
\left\{
    \begin{array}{l}
  \widehat{\mathcal{M}}_{k}(t;P_{k},i)=\mathcal{M}(t;P,i)+\mathcal{L}_{l}(t;P,i)\widehat{\Theta}_{l}(t,i)+\widehat{\Theta}_{l}(t,i)^{\top}\mathcal{L}_{l}(t;P,i)^{\top}+\widehat{\Theta}_{l}(t,i)^{\top}\mathcal{N}_{ll}(t;P,i)\widehat{\Theta}_{l}(t,i),\\[2mm]
  \widehat{\mathcal{L}}_{k}(t;P_{k},i)=\mathcal{L}_{k}(t;P,i)+\widehat{\Theta}_{l}(t,i)^{\top}\mathcal{N}_{lk}(t;P,i),\\[2mm]
  \widehat{\mathcal{N}}_{k}(t;P_{k},i)=\mathcal{N}_{kk}(t;P,i),\quad t\in[0,T],\quad i\in\mathcal{S},\quad (k,l)\in\{(1,2),(2,1)\}.
    \end{array}
    \right.
\end{equation}
Substituting \eqref{ZLQ-MLN-relation} into \eqref{ZLQ-CDREs-k}-\eqref{ZLQ-Riccati-constraint-1}, one has
\begin{equation}\label{ZLQ-CDREs-1}
   \left\{
   \begin{aligned}
   0&=\dot{P}(t,i)+\mathcal{M}(t;P,i)-\mathcal{L}_{1}(t;P,i)\mathcal{N}_{11}(t;P,i)^{\dag}\mathcal{L}_{1}(t;P,i)^{\top}\\
   &\quad +\big[\mathcal{L}_{2}(t;P,i)-\mathcal{L}_{1}(t;P,i)\mathcal{N}_{11}(t;P,i)^{\dag}\mathcal{N}_{12}(t;P,i)\big]\widehat{\Theta}_{2}(t,i)\\
   &\quad +\widehat{\Theta}_{2}(t,i)^{\top}\big[\mathcal{L}_{2}(t;P,i)-\mathcal{L}_{1}(t;P,i)\mathcal{N}_{11}(t;P,i)^{\dag}\mathcal{N}_{12}(t;P,i)\big]^{\top}\\
   &\quad +\widehat{\Theta}_{2}(t,i)^{\top}\big[\mathcal{N}_{22}(t;P,i)-\mathcal{N}_{21}(t;P,i)\mathcal{N}_{11}(t;P,i)^{\dag}\mathcal{N}_{12}(t;P,i)\big]\widehat{\Theta}_{2}(t,i),\quad t\in[0,T],\\
   P&(T,i)=G(i),\quad i\in\mathcal{S},
   \end{aligned}
   \right.
\end{equation}
and
\begin{equation}\label{ZLQ-CDREs-1-constraint}
    \left\{
    \begin{aligned}
    &\mathcal{N}_{11}(t;P,i)\widehat{\Theta}_{1}(t,i)+\mathcal{N}_{12}(t;P,i)\widehat{\Theta}_{2}(t,i)+\mathcal{L}_{1}(t;P,i)^{\top}=0,\\
    &\mathcal{N}_{21}(t;P,i)\widehat{\Theta}_{1}(t,i)+\mathcal{N}_{22}(t;P,i)\widehat{\Theta}_{2}(t,i)+\mathcal{L}_{2}(t;P,i)^{\top}=0,\\
    & \mathcal{N}_{11}(t;P,i)\geq 0,\quad \mathcal{N}_{22}(t;P,i)\leq 0,\quad i\in\mathcal{S}.
    \end{aligned}
    \right.
\end{equation}
Clearly, the condition \eqref{ZLQ-CDREs-1-constraint} is consistent with \eqref{ZLQ-closed-loop-constraint}. Plugging \eqref{ZLQ-CDREs-1-constraint} into  CDREs \eqref{ZLQ-CDREs-1} and recalling the notations \eqref{ZLQ-notation-LN}, we can further obtain that
\begin{equation}
\begin{aligned}
0&=\dot{P}(t,i)+\mathcal{M}(t;P,i)+\mathcal{L}_{2}(t;P,i)\widehat{\Theta}_{2}(t,i)+\widehat{\Theta}_{2}(t,i)^{\top}\mathcal{L}_{2}(t;P,i)+\widehat{\Theta}_{2}(t,i)^{\top}\mathcal{N}_{22}(t;P,i)\widehat{\Theta}_{2}(t,i)\\
   &\quad- \big[\mathcal{N}_{12}(t;P,i)\widehat{\Theta}_{2}(t,i)+\mathcal{L}_{1}(t;P,i)^{\top}\big]^{\top}\mathcal{N}_{11}(t;P,i)^{\dag}\big[\mathcal{N}_{12}(t;P,i)\widehat{\Theta}_{2}(t,i)+\mathcal{L}_{1}(t;P,i)^{\top}\big]\\
   &=\dot{P}(t,i)+\mathcal{M}(t;P,i)-\big[\mathcal{N}_{21}(t;P,i)\widehat{\Theta}_{1}(t,i)+\mathcal{N}_{22}(t;P,i)\widehat{\Theta}_{2}(t,i)\big]^{\top}\widehat{\Theta}_{2}(t,i)\\
   &\quad-\widehat{\Theta}_{2}(t,i)^{\top}\big[\mathcal{N}_{21}(t;P,i)\widehat{\Theta}_{1}(t,i)+\mathcal{N}_{22}(t;P,i)\widehat{\Theta}_{2}(t,i)\big]-\widehat{\Theta}_{1}(t,i)^{\top}\mathcal{N}_{11}(t;P,i)\widehat{\Theta}_{1}(t,i)\\
   &=\dot{P}(t,i)+\mathcal{M}(t;P,i)-\widehat{\Theta}_{1}(t,i)^{\top}\mathcal{N}_{11}(t;P,i)\widehat{\Theta}_{1}(t,i)-\widehat{\Theta}_{1}(t,i)^{\top}\mathcal{N}_{12}(t;P,i)\widehat{\Theta}_{2}(t,i)\\
   &\quad -\widehat{\Theta}_{2}(t,i)^{\top}\mathcal{N}_{21}(t;P,i)\widehat{\Theta}_{1}(t,i)-\widehat{\Theta}_{2}(t,i)^{\top}\mathcal{N}_{22}(t;P,i)\widehat{\Theta}_{2}(t,i)\\
   &=\dot{P}(t,i)+\mathcal{M}(t;P,i)-\widehat{\Theta}(t,i)^{\top}\mathcal{N}(t;P,i)\widehat{\Theta}(t,i)\\
   &=\dot{P}(t,i)+\mathcal{M}(t;P,i)-\mathcal{L}(t;p,i)\mathcal{N}(t;p,i)^{\dag}\mathcal{L}(t;p,i)^{\top}.
\end{aligned}
\end{equation}
The last equation follows from item (iii) in Remark \ref{rmk-Pseudoinverse}.

On the other hand, the condition \eqref{ZLQ-eta-k-constraint} can be rewritten as:
\begin{equation}\label{ZLQ-eta-k-constraint-2}
    \begin{aligned}
    &\mathcal{N}_{11}(P,\alpha)\widehat{\nu}_{1}+\mathcal{N}_{12}(P,\alpha)\widehat{\nu}_{2}+\widetilde{\rho}_{1}=0,\\
    &\mathcal{N}_{21}(P,\alpha)\widehat{\nu}_{1}+\mathcal{N}_{22}(P,\alpha)\widehat{\nu}_{2}+\widetilde{\rho}_{2}=0.
    \end{aligned}
\end{equation}
where
\begin{equation}
     \widetilde{\rho}_{k}\triangleq B_{k}(\alpha)^{\top}\eta_{k}+D_{k}(\alpha)^{\top}\zeta_{k}+D_{k}(\alpha)^{\top}P(\alpha)\sigma_{k}+\rho_{k},\quad k=1,2.
\end{equation}
Next, we shall verify that the solutions $(\eta_{1},\zeta_{1},\mathbf{z}_{1})$   and $(\eta_{2},\zeta_{2},\mathbf{z}_{2})$ to BSDE \eqref{ZLQ-eta-k} for $k=1,2$ are consistent, which further solves the BSDE \eqref{eta}. To this end, let
$$\widehat{\eta}(\cdot)=\eta_{1}(\cdot)-\eta_{2}(\cdot),\quad \widehat{\zeta}(\cdot)=\zeta_{1}(\cdot)-\zeta_{2}(\cdot),\quad
\widehat{\mathbf{z}}(\cdot)=\mathbf{z}_{1}(\cdot)-\mathbf{z}_{2}(\cdot).$$
%It follows from that  BSDE \eqref{ZLQ-eta-k} that $(\widehat{\eta}(\cdot),\widehat{\zeta}(\cdot),\widehat{\mathbf{z}}(\cdot))$ such that
Then,
\begin{equation}\label{hat-eta}
\left\{
\begin{aligned}
& d\widehat{\eta}(t)=-\Lambda(t)dt+\widehat{\zeta}(t)dW(t)+\widehat{\mathbf{z}}(t)\cdot d\mathbf{\widetilde{N}}(t),\quad t\in[0,T],  \\
&\widehat{\eta}(T)=0,
\end{aligned}
\right.
\end{equation}
where
\begin{align*}
\Lambda&=\big[\widehat{A}_{1}(\alpha)^{\top}-\widehat{\mathcal{L}}_{1}(P,\alpha)\widehat{\mathcal{N}}_{1}(P,\alpha)^{\dag}B_{1}(\alpha)^{\top}\big]\eta_{1}
        +\big[\widehat{C}_{1}(\alpha)^{\top}-\widehat{\mathcal{L}}_{1}(P,\alpha)\widehat{\mathcal{N}}_{1}(P,\alpha)^{\dag}D_{1}(\alpha)^{\top}\big]\zeta_{1}\\
        &\quad+\big[\widehat{C}_{1}(\alpha)^{\top}-\widehat{\mathcal{L}}_{1}(P,\alpha)\widehat{\mathcal{N}}_{1}(P,\alpha)^{\dag}D_{1}(\alpha)^{\top}\big]P(\alpha)\widehat{\sigma}_{1}-
       \widehat{\mathcal{L}}_{1}(P,\alpha)\widehat{\mathcal{N}}_{1}(P,\alpha)^{\dag}\widehat{\rho}_{1}+P(\alpha)\widehat{b}_{1}+\widehat{q}_{1}\\
       &\quad -\Big\{\big[\widehat{A}_{2}(\alpha)^{\top}-\widehat{\mathcal{L}}_{2}(P,\alpha)\widehat{\mathcal{N}}_{2}(P,\alpha)^{\dag}B_{2}(\alpha)^{\top}\big]\eta_{2}
    +\big[\widehat{C}_{2}(\alpha)^{\top}-\widehat{\mathcal{L}}_{2}(P,\alpha)\widehat{\mathcal{N}}_{2}(P,\alpha)^{\dag}D_{2}(\alpha)^{\top}\big]\zeta_{2}\\
        &\quad+\big[\widehat{C}_{2}(\alpha)^{\top}-\widehat{\mathcal{L}}_{2}(P,\alpha)\widehat{\mathcal{N}}_{2}(P,\alpha)^{\dag}D_{2}(\alpha)^{\top}\big]P(\alpha)\widehat{\sigma}_{2}
       -\widehat{\mathcal{L}}_{2}(P,\alpha)\widehat{\mathcal{N}}_{2}(P,\alpha)^{\dag}\widehat{\rho}_{2}+P(\alpha)\widehat{b}_{2}+\widehat{q}_{2}\Big\}\\
    &=A(\alpha)^{\top}\widehat{\eta}+C(\alpha)^{\top}\widehat{\zeta}+\widehat{\Theta}_{2}(\alpha)^{\top}\big[B_{2}(\alpha)^{\top}\eta_{1}+D_{2}(\alpha)^{\top}\zeta_{1}+D_{2}(\alpha)^{\top}P(\alpha)\sigma+\rho_{2}\big]\\
    &\quad -\widehat{\Theta}_{1}(\alpha)^{\top}\big[B_{1}(\alpha)^{\top}\eta_{2}+D_{1}(\alpha)^{\top}\zeta_{2}+D_{1}(\alpha)^{\top}P(\alpha)\sigma+\rho_{1}\big]+\big[\widehat{\Theta}_{2}(\alpha)^{\top}\mathcal{N}_{22}(P,\alpha)+\mathcal{L}_{2}(P,\alpha)\big]\widehat{\nu}_{2}\\
    &\quad-\big[\widehat{\Theta}_{1}(\alpha)^{\top}\mathcal{N}_{11}(P,\alpha)+\mathcal{L}_{1}(P,\alpha)\big]\widehat{\nu}_{1}-\big[\mathcal{L}_{1}(P,\alpha)+\widehat{\Theta}_{2}(\alpha)^{\top}\mathcal{N}_{21}(P,\alpha)\big]\mathcal{N}_{11}(P,\alpha)^{\dag}\big[\mathcal{N}_{12}(P,\alpha)\widehat{\nu}_{2}+\widetilde{\rho}_{1}\big]\\
    &\quad +\big[\mathcal{L}_{2}(P,\alpha)+\widehat{\Theta}_{1}(\alpha)^{\top}\mathcal{N}_{12}(P,\alpha)\big]\mathcal{N}_{22}(P,\alpha)^{\dag}\big[\mathcal{N}_{21}(P,\alpha)\widehat{\nu}_{1}+\widetilde{\rho}_{2}\big]\\
    &=A(\alpha)^{\top}\widehat{\eta}+C(\alpha)^{\top}\widehat{\zeta}+\widehat{\Theta}_{2}(\alpha)^{\top}\big[B_{2}(\alpha)^{\top}\eta_{1}+D_{2}(\alpha)^{\top}\zeta_{1}+D_{2}(\alpha)^{\top}P(\alpha)\sigma+\rho_{2}\big]\\
    &\quad -\widehat{\Theta}_{1}(\alpha)^{\top}\big[B_{1}(\alpha)^{\top}\eta_{2}+D_{1}(\alpha)^{\top}\zeta_{2}+D_{1}(\alpha)^{\top}P(\alpha)\sigma+\rho_{1}\big]
    -\widehat{\Theta}_{1}(\alpha)^{\top}\mathcal{N}_{12}(P,\alpha)\widehat{\nu}_{2}\\
    &\quad+\widehat{\Theta}_{2}(\alpha)^{\top}\mathcal{N}_{21}(P,\alpha)\widehat{\nu}_{1}
    +\widehat{\Theta}_{1}(\alpha)^{\top}\mathcal{N}_{11}(P,\alpha)\mathcal{N}_{11}(P,\alpha)^{\dag}\big[\mathcal{N}_{12}(P,\alpha)\widehat{\nu}_{2}+\widetilde{\rho}_{1}\big]\\
    &\quad -\widehat{\Theta}_{2}(\alpha)^{\top}\mathcal{N}_{22}(P,\alpha)\mathcal{N}_{22}(P,\alpha)^{\dag}\big[\mathcal{N}_{21}(P,\alpha)\widehat{\nu}_{1}+\widetilde{\rho}_{2}\big]\\
    &=\big[A(\alpha)+B(\alpha)^{\top}\widehat{\Theta}(\alpha)\big]^{\top}\widehat{\eta}+\big[C(\alpha)+D(\alpha)^{\top}\widehat{\Theta}(\alpha)\big]^{\top}\widehat{\zeta}.
\end{align*}
Substituting the above result into \eqref{hat-eta} and by the unique solvability of linear BSDE, we have
$$(\widehat{\eta},\widehat{\zeta},\widehat{\mathbf{z}})=(0,0,\mathbf{0}),$$
which implies that
\begin{equation}\label{eta-zeta-z}
  (\eta_{1}(\cdot),\zeta_{1}(\cdot),\widehat{\mathbf{z}}_{1}(\cdot))=(\eta_{2}(\cdot),\zeta_{2}(\cdot),\widehat{\mathbf{z}}_{2}(\cdot))\triangleq (\eta(\cdot),\zeta(\cdot),\widehat{\mathbf{z}}(\cdot)).
\end{equation}
Consequently,  the condition \eqref{ZLQ-eta-closed-constraint} follows from \eqref{ZLQ-eta-k-constraint-2}. Additionally, from \eqref{ZLQ-eta-k},  one has
\begin{equation}\label{ZLQ-eta-2}
       \left\{
      \begin{aligned}
        d\eta(t)&=-\Sigma(t)dt+\zeta(t) dW(t)+\mathbf{z}(t)\cdot d\mathbf{\widetilde{N}}(t),\quad t\in [0,T],\\
        \eta(T)&=g,
      \end{aligned}
      \right.
   \end{equation}
 where
 \begin{align*}
     \Sigma&=\big[\widehat{A}_{1}(\alpha)^{\top}-\widehat{\mathcal{L}}_{1}(P,\alpha)\widehat{\mathcal{N}}_{1}(P,\alpha)^{\dag}B_{1}(\alpha)^{\top}\big]\eta+\big[\widehat{C}_{1}(\alpha)^{\top}-\widehat{\mathcal{L}}_{1}(P,\alpha)\widehat{\mathcal{N}}_{1}(P_{k},\alpha)^{\dag}D_{1}(\alpha)^{\top}\big]\zeta\\
        &\quad+\big[\widehat{C}_{1}(\alpha)^{\top}-\widehat{\mathcal{L}}_{1}(P,\alpha)\widehat{\mathcal{N}}_{1}(P,\alpha)^{\dag}D_{1}(\alpha)^{\top}\big]P(\alpha)\widehat{\sigma}_{1}-\widehat{\mathcal{L}}_{1}(P,\alpha)\widehat{\mathcal{N}}_{1}(P,\alpha)^{\dag}\widehat{\rho}_{1}+P(\alpha)\widehat{b}_{1}+\widehat{q}_{1}\\
        &=\big[A(\alpha)+B_{2}(\alpha)\widehat{\Theta}_{2}(\alpha)\big]^{\top}\eta+\big[C(\alpha)+D_{2}(\alpha)\widehat{\Theta}_{2}(\alpha)\big]^{\top}\zeta+\big[C(\alpha)+D_{2}(\alpha)\widehat{\Theta}_{2}(\alpha)\big]^{\top}P(\alpha)\big[D_{2}(\alpha)\widehat{\nu}_{2}+\sigma\big]\\
        &\quad +P(\alpha)\big[B_{2}(\alpha)\widehat{\nu}_{2}+b\big]+q+\widehat{\Theta}_{2}(\alpha)^{\top}\rho_{2}
        +S_{2}(\alpha)^{\top}\widehat{\nu}_{2}+\widehat{\Theta}_{2}(\alpha)^{\top}R_{22}(\alpha)\widehat{\nu}_{2}\\
        &\quad -\big[\mathcal{L}_{1}(P,\alpha)+\widehat{\Theta}_{2}(\alpha)^{\top}\mathcal{N}_{21}(P,\alpha)\big]\mathcal{N}_{11}(P,\alpha)^{\dag}\big[\widetilde{\rho}_{1}+\mathcal{N}_{12}(P,\alpha)\widehat{\nu}_{2}\big]\\
        &=A(\alpha)^{\top}\eta+C(\alpha)^{\top}\zeta+C(\alpha)^{\top}P(\alpha)\sigma+P(\alpha)b+q+\widehat{\Theta}_{2}(\alpha)^{\top}\widetilde{\rho}_{2}+\mathcal{L}_{2}(P,\alpha)\widehat{\nu}_{2}+\widehat{\Theta}_{2}(\alpha)^{\top}\mathcal{N}_{22}(P,\alpha)\widehat{\nu}_{2}\\
        &\quad+\widehat{\Theta}_{1}(\alpha)^{\top}\big[\widetilde{\rho}_{1}+\mathcal{N}_{12}(P,\alpha)\widehat{\nu}_{2}\big]\\
        &=A(\alpha)^{\top}\eta+C(\alpha)^{\top}\zeta+C(\alpha)^{\top}P(\alpha)\sigma+P(\alpha)b+q+\widehat{\Theta}(\alpha)^{\top}\widetilde{\rho}.
 \end{align*}
 Noting that \eqref{ZLQ-closed-loop-constraint} implies
 $$\widehat{\Theta}(t,i)=-\mathcal{N}(t,P,i)^{\dag}\mathcal{L}(t,P,i)^{\top}+\big[I-\mathcal{N}(t,P,i)^{\dag}\mathcal{N}(t,P,i)\big]\Pi(t,i),\quad i\in\mathcal{S},$$
 for some $\Pi(\cdot,i)\in L^{2}(0,T;\mathbb{R}^{m\times n})$. Combining with the condition $\widetilde{\rho}\in\mathcal{R}(\mathcal{N}(P,\alpha))$, we have
 $$
 \widehat{\Theta}(\alpha)^{\top}\widetilde{\rho}=-\mathcal{N}(P,\alpha)^{\dag}\mathcal{L}(P,\alpha)^{\top}\widetilde{\rho}.
 $$
 Consequently, by the definition of $\widetilde{\rho}(\cdot)$ (see equation \eqref{eta-constraint}), we can further obtain that the $(\eta(\cdot),\zeta(\cdot),\widehat{\mathbf{z}}(\cdot))$ defined in \eqref{eta-zeta-z} also solves BSDE \eqref{eta}.

 Finally, by Theorem \ref{thm-SLQ-closed}, the value function of Problem (M-ZLQ) satisfies
 \begin{equation}\label{value-2}
 \begin{aligned}
     V(x,i)&=J(x,i;\widehat{\mathbf{\Theta}},\widehat{\nu})=\widehat{J}_{1}(x,i;\widehat{\mathbf{\Theta}}_{1},\widehat{\nu}_{1})+\mathbb{E}\int_{0}^{T}\big[\big<R_{22}(\alpha)\widehat{\nu}_{2}+2\rho_{2},\widehat{\nu}_{2}\big>\big]dt\\
     &=\mathbb{E}\Big\{\big<P(0,i)x,x\big>+2\big<\eta(0),x\big>+\int_{0}^{T}\big[\big<P(\alpha)\widehat{\sigma}_{1},\widehat{\sigma}_{1}\big>+2\big<\eta,\widehat{b}_{1}\big>+2\big<\zeta,\widehat{\sigma}_{1}\big>-\big<\mathcal{N}_{11}(P,\alpha)\widehat{\varrho}_{1},\widehat{\varrho}_{1}\big>\\
     &\quad+\big<R_{22}(\alpha)\widehat{\nu}_{2}+2\rho_{2},\widehat{\nu}_{2}\big>
     \big]dt\Big\}\\
     &=\mathbb{E}\Big\{\big<P(0,i)x,x\big>+2\big<\eta(0),x\big>+\int_{0}^{T}\big[\big<P(\alpha)\sigma,\sigma\big>+2\big<\eta,b\big>+2\big<\zeta,\sigma\big>-\big<\mathcal{N}_{11}(P,\alpha)\widehat{\nu}_{1},\widehat{\nu}_{1}\big>\\
     &\quad+\big<\mathcal{N}_{22}(\alpha)\widehat{\nu}_{2},\widehat{\nu}_{2}\big>+2\big<\widetilde{\rho}_{2},\widehat{\nu}_{2}\big>
     \big]dt\Big\}.
 \end{aligned}
 \end{equation}
 Note that
 \begin{align*}
     &\big<\mathcal{N}_{22}(\alpha)\widehat{\nu}_{2},\widehat{\nu}_{2}\big>-\big<\mathcal{N}_{11}(P,\alpha)\widehat{\nu}_{1},\widehat{\nu}_{1}\big>+2\big<\widetilde{\rho}_{2},\widehat{\nu}_{2}\big>\\
     =&\big<\mathcal{N}_{22}(\alpha)\widehat{\nu}_{2},\widehat{\nu}_{2}\big>-\big<\mathcal{N}_{11}(P,\alpha)\widehat{\nu}_{1},\widehat{\nu}_{1}\big>-2\big<\mathcal{N}_{21}(P,\alpha)\widehat{\nu}_{1}+\mathcal{N}_{22}(P,\alpha)\widehat{\nu}_{2},\widehat{\nu}_{2}\big>\\
     =&-\big<\mathcal{N}(P,\alpha)\widehat{\nu},\widehat{\nu}\big>\\
     =&-\big<\mathcal{N}(P,\alpha)\widetilde{\rho},\widetilde{\rho}\big>.
 \end{align*}
 Hence, we complete the proof by substituting the above equation into \eqref{value-2}.
\end{proof}

\begin{remark}\label{rmk-closed-zlq}\rm
Clearly, by the definition of closed-loop solvable for Problem (M-ZLQ) and the basic properties of pseudo-inverse, one can easily verify that Problem (M-ZLQ) is closed-loop solvable  if and only if:
\begin{description}
  \item[(i)] The CDREs \eqref{CDREs} admits a unique solution $\mathbf{P}(\cdot)\in\mathcal{D}\left(C(0,T;\mathbb{S}^{n})\right)$ such that
      \begin{equation}\label{ZLQ-closed-loop-constraint-2}
        \left\{
        \begin{aligned}
        &\mathcal{R}(\mathcal{L}(t;P,i)^{\top})\subseteq \mathcal{R}(\mathcal{N}(t;P,i)), \quad a.e.\text{ }t\in[0,T],\\
       &\mathcal{N}(\cdot;P,i)^{\dag}\mathcal{L}(\cdot;P,i)^{\top}\in L^{2}(0,T;\mathbb{R}^{m\times n}),\\
        &\mathcal{N}_{11}(\cdot;P,i)\geq 0,\quad \mathcal{N}_{22}(\cdot;P,i)\leq 0;
        \end{aligned}
        \right.
      \end{equation}
      \item[(ii)] The adapted solution $(\eta(\cdot),\zeta(\cdot),\mathbf{z}(\cdot))$ to BSDE \eqref{eta} satisfies the condition \eqref{eta-constraint}.
\end{description}
We point out that such a result can be degenerated to the case without regime switching jumps  (see, \citet{Sun.J.R.2014_NILQ}) by setting $\mathcal{S}=\{1\}$.
\end{remark}
The next result directly follows from  Theorem \ref{thm-open-closed}, and it provides a relation between open-loop and closed-loop solvabilities of the Problem (M-ZLQ).

\begin{corollary}\label{coro-closed-open}
 Let assumptions (A1)-(A2) and the convexity-concavity condition \eqref{convex-concave-condition} hold.  If Problem (M-ZLQ) is closed-loop solvable, then Problem (M-ZLQ) is also open-loop solvable. In this case, the outcome of the closed-loop equilibrium strategy is consistent with the closed-loop representation of the open-loop saddle point.
\end{corollary}

\section{The solvability of CDREs}\label{section-5} %\eqref{CDREs}
%under the condition \textbf{\eqref{uniformly-convex-concave-condition}}}\label{section-5}
%Before we give the solvability of CDREs \eqref{CDREs}, let's do some preparatory work first.
Now, we consider the following constrained CDREs
\begin{equation}\label{ZLQ-CCDREs}
   \left\{
    \begin{aligned}
    &\dot{P}(t,i)+\mathcal{M}(t;P,i)-\mathcal{L}(t;P,i)\mathcal{N}(t;P,i)^{-1}\mathcal{L}(t;P,i)^{\top}=0,\\
    &\mathcal{N}_{11}(t;P,i)\gg 0,\quad \mathcal{N}_{22}(t;P,i)\ll 0,\\
    &P(T,i)=G(i),\quad i\in\mathcal{S},
    \end{aligned}
    \right.
\end{equation}
Clearly, a solution $\mathbf{P}(\cdot)\in\mathcal{D}\left(C(0,T;\mathbb{R}^{n})\right)$ to the above constrained CDREs must be the solution to CDREs \eqref{CDREs} and satisfy the condition \eqref{ZLQ-closed-loop-constraint-2}. In this section, we shall study the solvability of CDREs \eqref{ZLQ-CCDREs}  under the uniform convexity-concavity condition \eqref{uniformly-convex-concave-condition}.

Let
\begin{equation}\label{convex-concave-cost}
\begin{aligned}
  &\widetilde{J}_{1}(x,i;u_{1})\triangleq \mathbb{E}\Big\{\int_{0}^{T}\left<\left(\begin{matrix}
Q(\alpha) & S_{1}(\alpha)^{\top}\\
S_{1}(\alpha) & R_{11}(\alpha)\end{matrix}\right)\left(\begin{matrix}
\widetilde{X}_{1}\\u_{1}\end{matrix}\right),
\left(\begin{matrix}
\widetilde{X}_{1}\\u_{1}\end{matrix}\right)\right>dt
 +\big<G(\alpha(T))\widetilde{X}_{1}(T), \widetilde{X}_{1}(T)\big>\Big\},\\
 &\widetilde{J}_{2}(x,i;u_{2})\triangleq \mathbb{E}\Big\{\int_{0}^{T}\left<\left(\begin{matrix}
Q(\alpha) & S_{2}(\alpha)^{\top}\\
S_{2}(\alpha) & R_{22}(\alpha)\end{matrix}\right)\left(\begin{matrix}
\widetilde{X}_{2}\\u_{2}\end{matrix}\right),
\left(\begin{matrix}
\widetilde{X}_{2}\\u_{2}\end{matrix}\right)\right>dt
 +\big<G(\alpha(T))\widetilde{X}_{2}(T), \widetilde{X}_{2}(T)\big>\Big\},
 \end{aligned}
\end{equation}
with $\widetilde{X}_{1}(\cdot)\equiv X^{0}(\cdot;x,i,u_{1},0)$ and $\widetilde{X}_{2}(\cdot)\equiv X^{0}(\cdot;x,i,0,u_{2})$.
%We denote the SLQ control problem with state process $\widetilde{X}_{k}(\cdot)$ and cost functional $\widetilde{J}_{k}(x,i;u_{k})$ as Problem (M-SLQ)$_{k}$.
Then the condition \eqref{uniformly-convex-concave-condition} can be rewritten as
\begin{equation}\label{uniformly-convex-concave-condition-2}
  \left\{
  \begin{array}{lll}
  \widetilde{J}_{1}(0,i;u_{1})\geq \lambda \mathbb{E}\int_{0}^{T}\big|u_{1}(t)\big|^{2}dt,&\forall u_{1}(\cdot)\in\mathcal{U}_{1},&\forall i\in\mathcal{S},\\[2mm]
  \widetilde{J}_{2}(0,i;u_{2})\leq -\lambda \mathbb{E}\int_{0}^{T}\big|u_{2}(t)\big|^{2}dt,&\forall u_{2}(\cdot)\in\mathcal{U}_{2},&\forall i\in\mathcal{S}.
  \end{array}
  \right.
\end{equation}
By Lemma \ref{thm-SLQ-uniform-convex}, we can immediately obtain that the following CDREs
\begin{equation}\label{section-5-CDREs-k}
   \left\{
    \begin{aligned}
    &\dot{P}_{k}(t,i)+\mathcal{M}(t;P_{k},i)-\mathcal{L}_{k}(t;P_{k},i)\mathcal{N}_{kk}(t;P_{k},i)^{-1}\mathcal{L}_{k}(t;P_{k},i)^{\top}=0,\\
    &P_{k}(T,i)=G(i),\quad i\in\mathcal{S},
    \end{aligned}
    \right.
\end{equation}
admits a unique solution $\mathbf{P_{k}}(\cdot)\in\mathcal{D}(C(0,T;\mathbb{S}^{n}))$ such that
$$
(-1)^{k+1}\mathcal{N}_{kk}(t;P_{k},i)\geq 0,\quad i\in\mathcal{S},\quad k=1,2.
$$

The following comparison result establishes the relation between $\mathbf{P}(\cdot)$ and $\mathbf{P_{k}}(\cdot)$ for $k=1,2$.

\begin{proposition}\label{prop-comparison}
Let assumptions (A1)-(A2) and condition \eqref{uniformly-convex-concave-condition} hold. Suppose that $\mathbf{P}(\cdot)$ is a solution of CDREs \eqref{ZLQ-CCDREs} over some interval $[\theta,\tau]\subseteq [0,T]$ with terminal condition replaced by $P(\tau,i)=H(i)\in\mathbb{S}^{n}$ for $i\in\mathcal{S}$. If $P_{1}(\tau,i)\leq H(i)\leq P_{2}(\tau,i)$ for every $i\in\mathcal{S}$, then
$$P_{1}(t,i)\leq P(t,i)\leq P_{2}(t,i),\quad t\in[\theta,\tau],\quad i\in\mathcal{S}.$$
\end{proposition}
\begin{proof}
Let
$$\Xi(t,i)\triangleq P(t,i)-P_{1}(t,i),\quad t\in[\theta,\tau].$$
  Then $\Xi(\tau,i)\geq 0$ and
  \begin{align*}
0&=\dot{\Xi}(t,i)+\Xi(t,i)A(t,i)+A(t,i)^{\top}\Xi(t,i)+C(t,i)^{\top}\Xi(t,i)C(t,i)+\sum_{j=1}^{L}\pi_{ij}(t)\Xi(t,j)\\
&\quad+\big[\mathcal{L}_{1}(t;P_{1},i)\mathcal{N}_{11}(t;P_{1},i)^{-1}\mathcal{L}_{1}(t;P_{1},i)^{\top}-\Delta(t;P,i)\big]
-\mathcal{L}_{1}(t;P,i)\mathcal{N}_{11}(t;P,i)^{-1}\mathcal{L}_{1}(t;P,i)^{\top},
  \end{align*}
  where
%  Note that
%  \begin{align*}
%   &\quad \mathcal{L}(t;P,i)\mathcal{N}(t;P,i)^{-1}\mathcal{L}(t;P,i)^{\top}\\
%   &=\left[\begin{matrix}\mathcal{L}_{1}(t;P,i) & \mathcal{L}_{2}(t;P,i)\end{matrix}\right]
%   \left[\begin{matrix}
%   \mathcal{N}_{11}(t;P,i) &\mathcal{N}_{12}(t;P,i)\\
%   \mathcal{N}_{21}(t;P,i)&\mathcal{N}_{22}(t;P,i)\end{matrix}\right]^{-1}
%   \left[\begin{matrix}\mathcal{L}_{1}(t;P,i)^{\top} \\ \mathcal{L}_{2}(t;P,i)^{\top}\end{matrix}\right]\\
%   &=\mathcal{L}_{1}(t;P,i)\mathcal{N}_{11}(t;P,i)^{-1}\mathcal{L}_{1}(t;P,i)^{\top}+\Delta(t;P,i),
%  \end{align*}
%  where
  \begin{align*}
  \Delta(t;P,i)&=\mathcal{L}(t;P,i)\mathcal{N}(t;P,i)^{-1}\mathcal{L}(t;P,i)^{\top}-\mathcal{L}_{1}(t;P,i)\mathcal{N}_{11}(t;P,i)^{-1}\mathcal{L}_{1}(t;P,i)^{\top}\\
  &=\big[\mathcal{N}_{21}(t;P,i)\mathcal{N}_{11}(t;P,i)^{-1}\mathcal{L}_{1}(t;P,i)^{\top}
  -\mathcal{L}_{2}(t;P,i)^{\top}\big]^{\top}\\
  &\quad\times \big[\mathcal{N}_{22}(t;P,i)
  -\mathcal{N}_{21}(t;P,i)\mathcal{N}_{11}(t;P,i)^{-1}\mathcal{N}_{12}(t;P,i)\big]^{-1}\\
  &\quad \times\big[\mathcal{N}_{21}(t;P,i)\mathcal{N}_{11}(t;P,i)^{-1}\mathcal{L}_{1}(t;P,i)^{\top}
  -\mathcal{L}_{2}(t;P,i)^{\top}\big]\leq 0,\quad t\in[\theta,\tau],\quad i\in\mathcal{S}.
  \end{align*}
 It follows from $\mathcal{N}_{11}(\cdot;P_{1},i)\gg 0$ that
  $$
  \mathcal{L}_{1}(t;P_{1},i)\mathcal{N}_{11}(t;P_{1},i)^{-1}\mathcal{L}_{1}(t;P_{1},i)^{\top}-\Delta(t;P,i)\geq 0,\quad a.e.\text{ }t\in[\theta.\tau].
  $$
 In addition, by some straightforward calculations, one has
   \begin{align*}
 &\quad \mathcal{L}_{1}(t;P,i)\mathcal{N}_{11}(t;P,i)^{-1}\mathcal{L}_{1}(t;P,i)^{\top}\\
 &=\big[\Xi(t,i)B_{1}(t,i)+C(t,i)^{\top}\Xi(t,i)D_{1}(t,i)+\mathcal{L}_{1}(t;P_{1},i)\big]
 \big[\mathcal{N}_{11}(t;P_{1},i)+D_{1}(t,i)^{\top}\Xi(t,i)D_{1}(t,i)\big]^{-1}\\
&\quad\times \big[\Xi(t,i)B_{1}(t,i)+C(t,i)^{\top}\Xi(t,i)D_{1}(t,i)+\mathcal{L}_{1}(t;P_{1},i)\big]^{\top}.
   \end{align*}
Then by Lemma \ref{thm-SLQ-uniform-convex}, we obtain that $\Xi(t,i)\geq 0$ for all $t\in[\theta,\tau]$
 and $i\in\mathcal{S}$, which is equivalent to $P_{1}(t,i)\leq P(t,i)$. One can prove that $P(t,i)\leq P_{2}(t,i)$
 for all $t\in[\theta,\tau]$ and $i\in\mathcal{S}$ in a similar manner.
\end{proof}

The following result provides the local solvability of CDREs \eqref{ZLQ-CCDREs}.

\begin{proposition}\label{prop-local-solvability}
  Let assumptions (A1)-(A2) and condition \eqref{uniformly-convex-concave-condition} hold, and let $\mathbf{P_{k}}(\cdot)$ be the solution of CDREs \eqref{section-5-CDREs-k}. For $\tau\in(0,T]$ and $\mathbf{H}\in\mathcal{D}(\mathbb{S}^{n})$, if
  $$P_{1}(\tau,i)\leq H(i)\leq P_{2}(\tau,i),\quad i\in\mathcal{S},$$
  then there exist a constant $\varepsilon>0$ such that the CDREs:
  \begin{equation}\label{ZLQ-CCDREs-2}
   \left\{
    \begin{aligned}
    &\dot{P}(t,i)+\mathcal{M}(t;P,i)-\mathcal{L}(t;P,i)\mathcal{N}(t;P,i)^{-1}\mathcal{L}(t;P,i)^{\top}=0,\\
    &P(\tau,i)=H(i),\quad i\in\mathcal{S},
    \end{aligned}
    \right.
\end{equation}
admits a solution $\mathbf{P}(\cdot)$ on the interval $[\tau-\varepsilon,\tau]$, which also satisfies the condition
\begin{equation}\label{ZLQ-CCDREs-2-condition}
  \mathcal{N}_{11}(t;P,i)\gg 0,\quad \mathcal{N}_{22}(t;P,i)\ll 0,\quad t\in[\tau-\varepsilon,\tau],\quad i\in\mathcal{S}.
\end{equation}
\end{proposition}

\begin{proof}
By a slight abuse of notation, let
\begin{align*}
&H=\text{diag}\{H(1),H(2),\cdots,H(L)\},\\
&A(\cdot)=\text{diag}\{A(\cdot,1)+\frac{1}{2}\pi_{11}(\cdot)I, A(\cdot,2)+\frac{1}{2}\pi_{22}(\cdot)I, \cdots , A(\cdot,L)+\frac{1}{2}\pi_{LL}(\cdot)I\},\\
&\Lambda(\cdot)=\text{diag}\{\Lambda(\cdot,1), \cdots,\Lambda(\cdot,L),\},\quad \Lambda=B_{k},\text{ }B,\text{ }C,\text{ }D_{k},\text{ }D,\text{ }Q,\text{ }S_{k},\text{ }S,\text{ }R_{kl},\text{ }R,\text{ }P_{k},\text{ }P,\quad k,l\in\{1,2\},
\end{align*}
and for any $P\in\mathbb{S}^{nL}$,
\begin{align*}
  &\mathcal{L}_{k}(\cdot,P) = PB_{k}(\cdot)+C(\cdot)^{\top}PD(\cdot)+S_{k}(\cdot)^{\top}, \quad
  \mathcal{N}_{kl}(\cdot,P)=  D_{k}(\cdot)^{\top}PD_{l}(\cdot)+R_{kl}(\cdot),\\
  & \mathcal{L}(\cdot,P) = PB(\cdot)+C(\cdot)^{\top}PD(\cdot)+S(\cdot)^{\top}, \quad
  \mathcal{N}(\cdot,P)=  D(\cdot)^{\top}PD(\cdot)+R(\cdot),
\end{align*}
Then the CDREs \eqref{ZLQ-CCDREs-2} can be rewritten as
\begin{equation}\label{ZLQ-CCDREs-3}
\left\{
\begin{aligned}
-\dot{P}(t)&=P(t)A(t)+A(t)P(t)+C(t)^{\top}P(t)C(t)+Q(t)+\sum_{j=1}^{L-1}\mathcal{T}_{j}(t)P(t)\mathcal{T}_{j}(t)^{\top}\\
&\quad -\mathcal{L}(t,P(t))\mathcal{N}(t,P(t))^{-1}\mathcal{L}(t,P(t))^{\top},\\
P(\tau)&=H,
\end{aligned}
\right.
\end{equation}
where
%\begin{align*}
%  \mathcal{L}(t,P) = P(t)B(t)+C(t)^{\top}P(t)D(t)+S(t)^{\top}, \quad
%  \mathcal{N}(t,P)=  D(t)^{\top}P(t)D(t)+R(t),
%\end{align*}
%and
$\mathcal{T}_{j}$ is the permutation matrix. For example, when $L = 4$,
\begin{align*}
&\mathcal{T}_{1}(t)=\left[\begin{matrix}
0     &     \sqrt{\pi_{12}(t)}I     &     0     &     0\\
0     &     0     &     \sqrt{\pi_{23}(t)}I     &     0\\
0     &     0     &     0     &     \sqrt{\pi_{34}(t)}I\\
\sqrt{\pi_{41}(t)}I    &     0     &    0     &     0\\
\end{matrix}\right]\\
&\mathcal{T}_{2}(t)=\left[\begin{matrix}
0     &     0     &     \sqrt{\pi_{13}(t)}I     &     0\\
0     &     0     &     0     &     \sqrt{\pi_{24}(t)}I\\
\sqrt{\pi_{31}(t)}I     &     0     &     0     &     0\\
0   &     \sqrt{\pi_{42}(t)}I     &    0     &     0\\
\end{matrix}\right]\\
&\mathcal{T}_{3}(t)=\left[\begin{matrix}
0     &     0     &     0     &     \sqrt{\pi_{14}(t)}I\\
\sqrt{\pi_{21}(t)}I     &     0     &     0     &     0\\
0     &     \sqrt{\pi_{32}(t)}I     &     0     &     0\\
0    &     0     &    \sqrt{\pi_{43}(t)}I     &     0\\
\end{matrix}\right].
\end{align*}

We define
\begin{align*}
&\mathcal{B}_{r}(H)=\big\{\Lambda\in \mathbb{S}^{nL}\big| |\Lambda-H|\leq r\big\},\quad r>0,\\
&F(t,P)=\mathcal{M}(t,P)-\mathcal{L}(t,P)\mathcal{N}(t,P)^{-1}\mathcal{L}(t,P)^{\top},
\end{align*}
where
$$\mathcal{M}(t,P)\triangleq PA(t)+A(t)P+C(t)^{\top}PC(t)+Q(t)+\sum_{j=1}^{L-1}\mathcal{T}_{j}(t)P\mathcal{T}_{j}(t)^{\top}.$$
In the following, we first prove that the generator function $F(t,P)$ is Lipschitz continuous in $P$ on $[\tau-\delta,\tau]\times \mathcal{B}_{r}(H)$ for some $\delta,\text{ } r>0$.% for proving that the CDREs \eqref{ZLQ-CCDREs-3} is locally solvable at $\tau$.
%To prove that the CDREs \eqref{ZLQ-CCDREs-3} is locally solvable at $\tau$, we only need to verify that the generator function $F(t,P)$ is Lipschitz continuous in $P$ on $[\tau-\delta,\tau]\times \mathcal{B}_{r}(H)$ for some $\delta,\text{ } r>0$. where $\mathcal{B}_{r}(H)$ is the closed ball in $\mathbb{S}^{nL}$ with center $H$ and radius $r$.

By Lemma \ref{thm-SLQ-uniform-convex}, one can select a constant $\beta>0$ such that
\begin{equation}\label{a.e.}
\mathcal{N}_{11}(t,P_{1}(t))\geq \beta I_{m_{1}L},\quad \mathcal{N}_{22}(t,P_{2}(t))\leq -\beta I_{m_{2}L},\quad a.e.\text{ }t\in[0,T].
\end{equation}
Since the solvability of CDREs \eqref{ZLQ-CCDREs-3} is not affected by the changed values of the coefficients on a zero Lebesgue measure set, we assume that \eqref{a.e.} holds for all $t\in[0,T]$ without loss of generality. Now, we denote $||D||_{\infty}$ as the essential supremum of $D(\cdot)\in L^{\infty}(0,T;\mathbb{R}^{nL\times mL})$ and set
$$
r=\frac{\beta}{4(||D||_{\infty}^{2}+1)}.
$$
Then by the continuities of $P_{1}(\cdot)$ and $P_{2}(\cdot)$,  we can select a small $\delta>0$ such that
$$ \big| P_{k}(\tau)-P_{k}(t) \big| \leq r,\quad t\in[\tau-\delta,\tau],\quad k=1,2.
$$
Consequently, for any $t\in[\tau-\delta,\tau]$  and $M\in \mathcal{B}_{r}(H)$, we have
\begin{equation}\label{p-1}
\begin{aligned}
\mathcal{N}_{11}(t,M)&=R_{11}(t)+D_{1}(t)^{\top}HD_{1}(t)+D_{1}(t)^{\top}(M-H)D_{1}(t)\\
&\geq R_{11}(t)+D_{1}(t)^{\top}P(\tau)D_{1}(t)-||D||_{\infty}^{2}\big| M-H\big| I_{m_{1}L}\\
&\geq \mathcal{N}_{11}(t,P_{1})-||D||_{\infty}^{2}\big| P_{1}(\tau)-P_{1}(t)\big| I_{m_{1}L}-||D||_{\infty}^{2}\big| M-H\big| I_{m_{1}L}\\
&\geq \beta I_{m_{1}L}-2r||D||_{\infty}^{2} I_{m_{1}L} \\
&\geq \frac{\beta}{2} I_{m_{1}L}.
\end{aligned}
\end{equation}
Similarly, one can further verify that
\begin{equation}\label{p-2}
\mathcal{N}_{22}(t,M)\leq -\frac{\beta}{2} I_{m_{2}L},\quad t\in[\tau-\delta,\tau],\quad M\in \mathcal{B}_{r}(H).
\end{equation}
Therefore, $\mathcal{N}_{11}(t,P)$ and $\mathcal{N}_{22}(t,P)$ are invertible for every $(t,P)\in[\tau-\delta,\tau]\times\mathcal{B}_{r}(H)$ with
$$
\big|\mathcal{N}_{11}(t,P)\big|\leq \frac{2}{\beta}\sqrt{m_{1}L},\quad
\big|\mathcal{N}_{22}(t,P)\big|\leq \frac{2}{\beta}\sqrt{m_{2}L}.
$$
Let
$$\Phi(t,P)\triangleq \mathcal{N}_{22}(t,P)-\mathcal{N}_{21}(t,P)\mathcal{N}_{11}(t,P)^{-1}\mathcal{N}_{12}(t,P) \quad \forall (t,P)\in[\tau-\delta,\tau]\times\mathcal{B}_{r}(H).$$
Then
$$\Phi(t,P)\leq \mathcal{N}_{22}(t,P)\quad
 \text{and}\quad
\big|\Phi(t,P)^{-1}\big|\leq \frac{2}{\beta}\sqrt{m_{2}L}.$$

On the other hand, due to the boundness of the coefficients in the state equation and performance functional, we can choose a constant $\gamma>0$ such that
$$\big|\mathcal{N}(t,P)\big|+\big|\mathcal{L}(t,P)\big|\leq \gamma, \quad \forall (t,P)\in[\tau-\delta,\tau]\times\mathcal{B}_{r}(H).$$
For simplicity, let $\gamma>0$ represent a generic constant that can be different from line to line and independent of $(t,P)\in[\tau-\delta,\tau]\times\mathcal{B}_{r}(H)$. Then it follows from Lemma \ref{lem-matrix} that
\begin{align*}
\big|\mathcal{N}(t,P)^{-1}\big|^{2}&=\big|\mathcal{N}_{11}(t,P)^{-1}+
\mathcal{N}_{11}(t,P)^{-1}\mathcal{N}_{12}(t,P)\Phi(t,P)^{-1}\mathcal{N}_{21}(t,P)\mathcal{N}_{11}(t,P)^{-1}\big|^{2}\\
&\quad +2\big|\mathcal{N}_{11}(t,P)^{-1}\mathcal{N}_{12}(t,P)\Phi(t,P)^{-1}\big|^{2}+\big|\Phi(t,P)^{-1}\big|^{2}\\
&\leq \gamma, \quad \forall (t,P)\in[\tau-\delta,\tau]\times\mathcal{B}_{r}(H).
\end{align*}
Noting that for any $(t,P),\text{ }(t,M)\in[\tau-\delta,\tau]\times\mathcal{B}_{r}(H)$,
\begin{align*}
&\big|\mathcal{M}(t,P)-\mathcal{M}(t,M)\big|+\big|\mathcal{L}(t,P)-\mathcal{L}(t,M)\big| \leq \gamma \big|P-M\big|,\\
&\begin{aligned}
\mathcal{N}(t,P)^{-1}-\mathcal{N}(t,M)^{-1}&=\mathcal{N}(t,P)^{-1}D(t)^{\top}(M-P)D(t)\mathcal{N}(t,M)^{-1}\\
&\leq \gamma \big|P-M\big|,
\end{aligned}
\end{align*}
and
\begin{align*}
&\quad \big|\mathcal{L}(t,P)\mathcal{N}(t,P)^{-1}\mathcal{L}(t,P)^{\top}
-\mathcal{L}(t,M)\mathcal{N}(t,M)^{-1}\mathcal{L}(t,M)^{\top}\big|\\
&\leq \big|\big(\mathcal{L}(t,P)-\mathcal{L}(t,M)\big)\mathcal{N}(t,P)^{-1}\mathcal{L}(t,P)^{\top}\big|
+\big|\mathcal{L}(t,M)\mathcal{N}(t,P)^{-1}\big(\mathcal{L}(t,P)-\mathcal{L}(t,M)\big)^{\top}\big|\\
&\quad+\big|\mathcal{L}(t,M)\big(\mathcal{N}(t,P)^{-1}-\mathcal{N}(t,M)^{-1}\big)\mathcal{L}(t,M)^{\top}\big|\\
&\leq \gamma\big|P-M\big|,
\end{align*}
which implies the Lipschitz continuity of $F(t,P)$ on the $[\tau-\delta,\tau]\times\mathcal{B}_{r}(H)$. By the contraction mapping theorem, we further derive the solvability of CDREs \eqref{ZLQ-CCDREs-3} (or equivalently, CDREs \eqref{ZLQ-CCDREs-2}) on a small interval $[\tau-\varepsilon,\tau]$. Moreover, the condition \eqref{ZLQ-CCDREs-2-condition} follows from the \eqref{p-1}-\eqref{p-2} directly. This completes the proof.
\end{proof}

Now, we provide the main result of this section.

\begin{theorem}\label{thm-ZLQ-CDREs-solvability}
  Let assumptions (A1)-(A2) and condition \eqref{uniformly-convex-concave-condition}. Then CDREs \eqref{ZLQ-CCDREs} admits a solution over $[0,T]$.
\end{theorem}
\begin{proof}
 By Proposition \ref{prop-local-solvability}, we know that the CDREs \eqref{ZLQ-CCDREs} are locally solvable at $T$.  In the following, we shall prove that the solvable interval can be extended to $[0,T]$.

 Without loss of generality, we first assume that the CDREs \eqref{ZLQ-CCDREs} are solvable on the interval $(\tau,T]$. Then by Proposition \ref{prop-comparison}, we have
 $$
 P_{1}(t,i)\leq P(t,i)\leq P_{2}(t,i),\quad t\in(\tau,T],\quad i\in\mathcal{S},
 $$
 which implies that $\mathbf{P}(\cdot)$ is bounded on the interval $(\tau,T]$ and
 \begin{align*}
  P(t,i)=G(i)+\int_{t}^{T}\big[\mathcal{M}(s;P,i)-\mathcal{L}(s;P,i)\mathcal{N}(s;P,i)^{-1}\mathcal{L}(s;P,i)^{\top}\big]ds,\quad t\in(\tau,T],\quad i\in\mathcal{S},
 \end{align*}
 is uniformly continuous. Hence, there exist a finite $\mathbf{P}(\tau)=\big[P(\tau,1),\cdots,P(\tau,L)\big]$ such that
 $$P(\tau,i)=\lim_{t\downarrow \tau}P(t,i),\quad i\in\mathcal{S},$$
 and we can extend the solvable  interval $(\tau,T]$ to $[\tau,T]$.

 Suppose that $[\tau,T]$  is the maximal interval on which a solution $\mathbf{P}(\cdot)$ of \eqref{ZLQ-CCDREs} exists. Note that
 $$
 P_{1}(\tau,i)\leq P(\tau,i)\leq P_{2}(\tau,i),\quad  i\in\mathcal{S}.
 $$
 If $\tau>0$, then we can further extend the solvable  interval $[\tau,T]$ to $[\tau-\varepsilon,T]$ for some small $\varepsilon>0$  by applying Proposition \ref{prop-local-solvability} again. This contradicts the maximality of $[\tau,T]$. So we must have $\tau=0$.
\end{proof}

Clearly, if the CDREs \eqref{ZLQ-CCDREs} admits a solution $\mathbf{P}(\cdot)\in\mathcal{D}(C(0,T;\mathbb{S}^{n}))$, then
%it must satisfies the condition \eqref{ZLQ-closed-loop-constraint-2}. Additionally,
 the associated BSDE \eqref{eta} also admits a solution $(\eta(\cdot),\zeta(\cdot),\mathbf{z}(\cdot))$ satisfies condition \eqref{ZLQ-eta-closed-constraint}. Combining with Theorem \ref{thm-ZLQ-closed-loop},  Remark \ref{rmk-closed-zlq}, and Corollary \ref{coro-closed-open}, we have the following result directly.
\begin{corollary}\label{coro-ZLQ-solvability}
Let assumptions (A1)-(A2) and condition \eqref{uniformly-convex-concave-condition} hold. Then, the Problem (M-ZLQ) is both open-loop solvable and closed-loop solvable. In this case, the closed-loop optimal strategy $(\widehat{\mathbf{\Theta}}(\cdot),\widehat{\nu}(\cdot))\in \mathcal{D}\left(L^{2}(0,T;\mathbb{R}^{m\times n})\right)\times L_{\mathbb{F}}^{2}(0,T;\mathbb{R}^{m})$ of Problem (M-ZLQ) admits the following representation:
  \begin{equation}\label{ZLQ-closed-optimal-2}
      \left\{
      \begin{aligned}
      &\widehat{\Theta}(\cdot,i)=-\mathcal{N}(\cdot,P,i)^{-1}\mathcal{L}(\cdot,P,i)^{\top},\quad i\in\mathcal{S},\\
      &\widehat{\nu}(\cdot)=-\mathcal{N}(\cdot,P,\alpha(\cdot))^{-1}\widetilde{\rho}(\cdot),
      \end{aligned}
      \right.
  \end{equation}
  where $\widetilde{\rho}(\cdot)$ is defined in \eqref{eta-constraint}, $\mathbf{P}(\cdot)$ is the solution to CDREs \eqref{ZLQ-CCDREs} and $(\eta(\cdot),\zeta(\cdot),\mathbf{z}(\cdot))$ is the solution of the following BSDE:
   \begin{equation}\label{ZLQ-eta}
  \left\{
      \begin{aligned}
        d\eta(t)&=-\big\{\big[A(\alpha)^{\top}-\mathcal{L}(P,\alpha)\mathcal{N}(P,\alpha)^{-1}B(\alpha)^{\top}\big]\eta+\big[C(\alpha)^{\top}-\mathcal{L}(P,\alpha)\mathcal{N}(P,\alpha)^{-1}D(\alpha)^{\top}\big]\zeta\\
        &\quad+\big[C(\alpha)^{\top}-\mathcal{L}(P,\alpha)\mathcal{N}(P,\alpha)^{-1}D(\alpha)^{\top}\big]P(\alpha)\sigma-\mathcal{L}(P,\alpha)\mathcal{N}(P,\alpha)^{-1}\rho+P(\alpha)b+q\big\}dt\\
        &\quad+\zeta dW(t)+\mathbf{z}\cdot d\mathbf{\widetilde{N}}(t),\quad t\in [0,T],\\
        \eta(T)&=m.
      \end{aligned}
      \right.
  \end{equation}
  In addition, the open-loop saddle point admits the following closed-loop representation
  \begin{equation}\label{ZLQ-open-loop-optimal-saddle-point-2}
    u^{*}(\cdot;x,i)=[u_{1}^{*}(\cdot;x,i)^{\top},u_{2}^{*}(\cdot;x,i)^{\top}]^{\top}
    =\widehat{\Theta}(\cdot,\alpha(\cdot))X(\cdot;x,i,\mathbf{\widehat{\Theta}},\widehat{\nu})+\widehat{\nu}(\cdot),
    \quad (x,i)\in\mathbb{R}^{n}\times\mathcal{S}.
  \end{equation}
\end{corollary}
\section{Numerical examples}\label{section-6}
In this section, we present a concrete example to illustrate the results in the previous sections. Without loss of generality, we suppose the state space of the Markov chain $\alpha$ is $\mathcal{S}:=\left\{1,2,3\right\}$ and the corresponding generator is given by
$$\Pi=\left[\begin{array}{ccc}
     \pi_{11}   &   \pi_{12}   &   \pi_{13}\\
    \pi_{21}   &   \pi_{22}   &   \pi_{23}\\
    \pi_{31}   &   \pi_{32}   &   \pi_{33}
\end{array}\right]
=\left[\begin{array}{ccc}
     -0.5   &   0.3   &   0.2\\
    0.2   &   -0.4   &   0.2\\
    0.3   &   0.2   &   -0.5
\end{array}\right].$$
For simplicity, we assume that both the state process and the control process are one-dimensional, and the coefficients in both the state equation and the performance functional depend only on the Markov chain $\alpha(\cdot)$.

Consider the following state equation and performance functional:
\begin{equation}\label{state-example}
\left\{
\begin{aligned}
dX(t)&=\left[A(\alpha)X+B_{1}(\alpha)u_{1}+B_{2}(\alpha)u_{2}\right]dt
+\left[C(\alpha)X(t)+D_{1}(\alpha)u_{1}+D_{2}(\alpha)u_{2}\right]dW(t),\\
   X(0)&=x,\quad \alpha(0)=i, \quad t\in[0,1],
\end{aligned}
\right.
\end{equation}
\begin{equation}\label{performance-functional-example}
\begin{aligned}
   J(x,i;u_{1},u_{2})
    \triangleq &\mathbb{E}\left\{\int_{0}^{1}
    \left<
    \left(
    \begin{matrix}
    Q(\alpha) & S_{1}(\alpha)^{\top} & S_{2}(\alpha)^{\top} \\
    S_{1}(\alpha) & R_{11}(\alpha) & R_{12}(\alpha) \\
    S_{2}(\alpha) & R_{21}(\alpha) & R_{22}(\alpha)
    \end{matrix}
    \right)
    \left(
    \begin{matrix}
    X \\
    u_{1} \\
    u_{2}
    \end{matrix}
   \right),
    \left(
    \begin{matrix}
    X\\
    u_{1} \\
    u_{2}
    \end{matrix}
    \right)
    \right>dt
   +\big<G(\alpha(1))X(1), X(1)\big>\right\},
  \end{aligned}
\end{equation}
where the coefficients are given by
\begin{table}[h!]
  \centering
\begin{tabular}{|c|c|c|c|c|c|c|}
  \hline
   $\alpha$ & $A(i)$ & $B_{1}(i)$ & $B_{2}(i)$ & $C(i)$ & $D_{1}(i)$ & $D_{2}(i)$ \\
    \hline
  $i=1$    & -1      &    1     &    -1     &     1     &       1     &      -1 \\
  $i=2$    & -1.5   &   -1     &     0     &     1     &      -1     &       0 \\
  $i=3$    & -1      &    0     &     1     &    -1     &       0     &       1 \\
  \hline
\end{tabular}
\end{table}

\noindent and
\begin{table}[h!]
  \centering
\begin{tabular}{|c|c|c|c|c|c|c|c|}
  \hline
   $\alpha$ & $Q(i)$ & $S_{1}(i)$ & $S_{2}(i)$ & $R_{11}(i)$ & $R_{12}(i)$ & $R_{22}(i)$ & $G(i)$\\
    \hline
  $i=1$    &    -1      &    1     &    -2     &     10     &       8      &       -12      &     -1 \\
  $i=2$    &     5      &    5     &     0     &      12     &      -4      &       -15      &      2 \\
  $i=3$    &     0      &    0     &     1     &      14     &       6      &       -13      &      0\\
  \hline
\end{tabular}
\end{table}

Clearly, the Problem (M-ZLQ) with the above coefficients satisfies the assumptions (A1)-(A2).
Next, we verify that the condition \eqref{uniformly-convex-concave-condition} also holds in this example. To this end, let
$X_{k}(\cdot)$ be the solution of the following SDE:
\begin{equation}\label{state-example-k}
\left\{
\begin{aligned}
dX(t)&=\left[A(\alpha)X+B_{k}(\alpha)u_{k}\right]dt
+\left[C(\alpha)X(t)+D_{k}(\alpha)u_{k}\right]dW(t),\quad t\in[0,1],\\
   X(0)&=0,\quad \alpha(0)=i.
\end{aligned}
\right.
\end{equation}
Then applying the Ito's rule to $\big<X_{k}(\cdot),X_{k}(\cdot)\big>$, we have
\begin{align*}
\mathbb{E}\big|X_{k}(t)\big|^{2}&=\mathbb{E}\int_{0}^{t}\big[\big<(A(\alpha)+A(\alpha)^{\top}
+C(\alpha)^{\top}C(\alpha))X_{k},X_{k}\big>
+2\big<(B_{k}(\alpha)+C(\alpha)^{\top}D_{k}(\alpha))u_{k},X_{k}\big>\\
&\quad +\big<D_{k}(\alpha)^{\top}D_{k}(\alpha)u_{k},u_{k}\big>\big]ds\\
&\leq\mathbb{E}\int_{0}^{t}\big[\big<\mathcal{A}(\alpha)X_{k},X_{k}\big>
 +\big<\mathcal{B}_{k}(\alpha)u_{k},u_{k}\big>\big]ds.
\end{align*}
where
\begin{align*}
&\mathcal{A}(i)=A(i)+A(i)^{\top}+C(i)^{\top}C(i)+I,\\
&\mathcal{B}_{k}(i)=[B_{k}(i)+C(i)^{\top}D_{k}(i)]^{\top}[B_{k}(i)+C(i)^{\top}D_{k}(i)]
+D_{k}(i)^{\top}D_{k}(i),\quad i\in\mathcal{S},\quad k=1,2,
\end{align*}
%Plugging the coefficients into the above equations, we have
are given by:
\begin{table}[h!]
  \centering
\begin{tabular}{|c|c|c|c|}
  \hline
   $\alpha$ & $\mathcal{A}(i)$ & $\mathcal{B}_{1}(i)$ & $\mathcal{B}_{2}(i)$ \\
    \hline
  $i=1$    &     0      &     5     &    5 \\
  $i=2$    &    -1      &    5     &     0  \\
  $i=3$    &     0      &    0     &     1 \\
  \hline
\end{tabular}
\end{table}

\noindent Consequently,
\begin{align*}
&\mathbb{E}\big|X_{k}(t)\big|^{2}\leq \int_{0}^{t}\big[ \mathbb{E}\big|X_{1}(s)\big|^{2}+5\mathbb{E}\big|u_{1}(s)\big|^{2}
\big]ds.
\end{align*}
By Gronwall's inequality, we can further obtain
\begin{align*}
&\mathbb{E}\big|X_{k}(t)\big|^{2}\leq \int_{0}^{t} 5e^{t-s}\mathbb{E}\big|u_{1}(s)\big|^{2}ds,
\end{align*}
which implies
\begin{align*}
\mathbb{E}\int_{0}^{1}\big|X_{k}(t)\big|^{2}dt&\leq \mathbb{E}\int_{0}^{1}\Big[\int_{0}^{t} 5e^{t-s}\big|u_{1}(s)\big|^{2}ds\Big] dt= \mathbb{E}\int_{0}^{1}\Big[\int_{s}^{1} 5e^{t-s}\big|u_{1}(s)\big|^{2}dt\Big] ds\\
&\leq 5(e-1) \mathbb{E}\int_{0}^{1} \big|u_{1}(t)\big|^{2} dt.
\end{align*}

On the other hand, applying It\^o's rule to $\big<G(\alpha(\cdot)X_{k}(\cdot)),X_{k}(\cdot)\big>$, one has
\begin{align*}
 J(0,i;u_{k},0)
    \triangleq &\mathbb{E}\int_{0}^{1}
    \left<
    \left(
    \begin{matrix}
    \mathcal{M}(\alpha) & \mathcal{L}_{k}(\alpha)\\
    \mathcal{L}_{k}(\alpha)^{\top} & \mathcal{N}_{k}(\alpha)
    \end{matrix}
    \right)
    \left(
    \begin{matrix}
    X_{k} \\
    u_{k}
    \end{matrix}
   \right),
    \left(
    \begin{matrix}
    X_{k}\\
    u_{k}
    \end{matrix}
    \right)
    \right>dt,
\end{align*}
where
$$
\begin{array}{l}
  \mathcal{M}(i)=G(i)A(i)+ A(i)^{\top}G(i)+C(i)^{\top}G(i)C(i)+Q(i)+\sum_{j=1}^{L}\pi_{ij}G(j)\\[2mm]
  \mathcal{L}_{k}(i)=G(i)B_{k}(i) +C(i)^{\top}G(i)D_{k}(i)+S_{k}(i)^{\top}\\[2mm]
 \mathcal{N}_{k}(i)=D_{k}(i)^{\top}G(i)D_{k}(i)+R_{kk}(i),\quad i\in\mathcal{S},\quad k=1,2
 \end{array}
$$
%Substituting the coefficients into the above equations, one has
satisfy
\begin{table}[h!]
  \centering
\begin{tabular}{|l|c|c|c|c|c|}
  \hline
   $\alpha$ & $\mathcal{M}(i)$ & $\mathcal{L}_{1}(i)$ & $\mathcal{L}_{2}(i)$ & $\mathcal{N}_{1}(i)$ & $\mathcal{N}_{2}(i)$ \\
    \hline
  $i=1$ & 1.1     & -1     &      0      &       11     & -11 \\
  $i=2$ & 0        & 1      &      0      &       13     & -15 \\
  $i=3$ & 0.1     & 0      &      1      &       14     & -12\\
  \hline
\end{tabular}
\end{table}

\noindent Hence,
\begin{align*}
 J(0,i;u_{1},0)
    &\geq\mathbb{E}\int_{0}^{1}\big[
    \big<\big(\mathcal{M}(\alpha)-\mathcal{L}_{1}(\alpha)\mathcal{L}_{1}(\alpha)^{\top}\big)X_{1}(t),X_{1}(t)\big>+\big<\big(\mathcal{N}_{1}(\alpha)-I\big)u_{1}(t),u_{1}(t)\big>\big]dt\\
    &\geq \mathbb{E}\int_{0}^{1}\big[
    -\big|X_{1}(t)\big|^{2}+10\big|u_{1}(t)\big|^{2}\big]dt\\
    &\geq [10-5(e-1)]\mathbb{E}\int_{0}^{1}\big|u_{1}(t)\big|^{2}dt,\\
 J(0,i;0,u_{2})
    &\leq\mathbb{E}\int_{0}^{1}\big[
    \big<\big(\mathcal{M}(\alpha)+\mathcal{L}_{2}(\alpha)\mathcal{L}_{2}(\alpha)^{\top}\big)X_{2}(t),X_{2}(t)\big>
    +\big<\big(\mathcal{N}_{2}(\alpha)+I\big)u_{2}(t),u_{2}(t)\big>\big]dt\\
    &\leq \mathbb{E}\int_{0}^{1}\big[
    -1.1\big|X_{2}(t)\big|^{2}-10\big|u_{1}(t)\big|^{2}\big]dt\\
    &\leq -[10-5.5(e-1)]\mathbb{E}\int_{0}^{1}\big|u_{1}(t)\big|^{2}dt.
\end{align*}
Noting that
$$10-5(e-1)>10-5.5(e-1)>0,$$
which implies that the condition \eqref{uniformly-convex-concave-condition} also holds in this example.

%Substituting the coefficients into CDREs \eqref{ZLQ-CCDREs}, we have
%\begin{equation}\label{CDREs-example}
%\left\{
% \begin{aligned}
%&0= \dot{P}(t,1)-\left[\begin{matrix}
% 2P(t,1)+1 & -2P(t,1)-2
% \end{matrix}\right]
% \left[\begin{matrix}
% P(t,1)+10 & -P(t,1)+8\\
% -P(t,1)+8 & P(t,1)-12
% \end{matrix}\right]^{-1}
% \left[\begin{matrix}
% 2P(t,1)+1 \\
%  -2P(t,1)-2
% \end{matrix}\right]\\
% &\qquad -P(t,1)-1-0.5P(t,1)+0.3P(t,2)+0.2P(t,3),\\
% &0= \dot{P}(t,2)-\left[\begin{matrix}
% -2P(t,2)+5 & 0
% \end{matrix}\right]
% \left[\begin{matrix}
% P(t,2)+12 & -4\\
% -4 & -15
% \end{matrix}\right]^{-1}
% \left[\begin{matrix}
% -2P(t,2)+5 \\
%  0
% \end{matrix}\right]\\
% &\qquad -2P(t,2)+5+0.2P(t,1)-0.4P(t,2)+0.2P(t,3),\\
%  &0= \dot{P}(t,3)-\left[\begin{matrix}
%0 & 1
% \end{matrix}\right]
% \left[\begin{matrix}
% 14 & 6\\
% 6 & P(t,3)-13
% \end{matrix}\right]^{-1}
% \left[\begin{matrix}
%0 \\
%  1
% \end{matrix}\right]\\
% &\qquad -P(t,3)+0.3P(t,1)+0.2P(t,2)-0.5P(t,3),\\
% &P(1,1)=-1,\quad P(1,2)=2,\quad P(1,3)=0.
% \end{aligned}
% \right.
%\end{equation}
It follows from Theorem \ref{thm-ZLQ-CDREs-solvability} that  the corresponding CDREs \eqref{ZLQ-CCDREs} for this example
admits a solution such that $\mathcal{N}_{11}(\cdot;P,i)\gg 0$ and $\mathcal{N}_{22}(\cdot;P,i)\ll 0$ for all $i\in\{1,2,3\}$. Using the well-known finite difference method, we can present the following numerical solution figure of $\left[P(\cdot,1), P(\cdot,2), P(\cdot,3)\right]$ for clearer visualization.
\begin{figure}[H]
  \centering
  \begin{tabular}{c}
    \includegraphics[width=200pt,height=150pt]{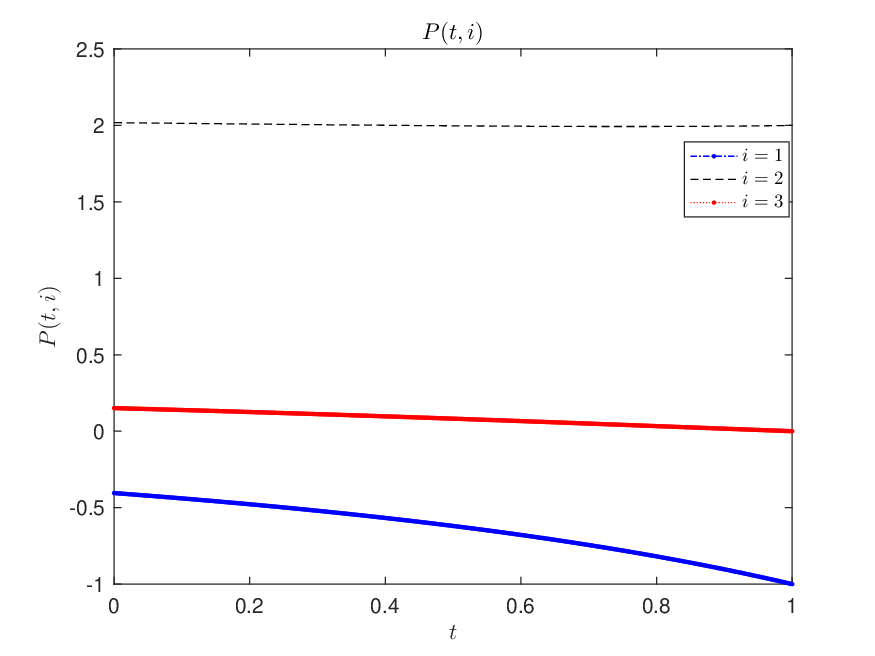}
  \end{tabular}
\end{figure}
Obviously,
%With the solution to \eqref{CDREs-example},
the unique solution to BSDE \eqref{ZLQ-eta} in this example is $(0,0,\mathbf{0})$. Therefore, by Corollary \ref{coro-ZLQ-solvability}, the open-loop saddle point is given by
$$u^{*}(\cdot;x,i)=[u_{1}^{*}(\cdot;x,i)^{\top},u_{2}^{*}(\cdot)^{\top}]^{\top}
=\widehat{\Theta}(\cdot;\alpha(\cdot))X^{*}(\cdot;x,i),\quad \forall(x,i)\in\mathbb{R}\times\{1,2,3\},$$
where (noting the notations defined in \eqref{notation-LMN})
$$
\widehat{\Theta}(\cdot,i)=[\widehat{\Theta}_{1}(\cdot,i)^{\top},\widehat{\Theta}_{2}(\cdot,i)^{\top}]^{\top}=-\mathcal{N}(\cdot,P,i)^{-1}\mathcal{L}(\cdot,P,i)^{\top},\quad i\in\{1,2,3\},
$$
and $X^{*}(\cdot;x,i)$ be the solution to the following SDE:
\begin{equation}\label{optimal-state-example}
\left\{
\begin{aligned}
dX^{*}(t)&=\left[A(\alpha(t))+B_{1}(\alpha(t))\widehat{\Theta}_{1}(t,\alpha(t))
+B_{2}(\alpha(t))\widehat{\Theta}_{2}(t,\alpha(t))\right]X^{*}(t)dt\\
&\quad+\left[C(\alpha(t))+D_{1}(\alpha(t))\widehat{\Theta}_{1}(t,\alpha(t))
+D_{2}(\alpha(t))\widehat{\Theta}_{2}(t,\alpha(t))\right]X^{*}(t)dW(t),\\
   X^{*}(0)&=x,\quad \alpha(0)=i, \quad t\in[0,1].
\end{aligned}
\right.
\end{equation}

\section*{Declarations}
\textbf{Conflict of interest:} The authors have not disclosed any competing interests.
%\newpage
\bibliography{references}
\bibliographystyle{abbrvnat}

\end{document}